\documentclass{ws-m3as}
\usepackage{hyperref}
\usepackage{braket,amsfonts}
\usepackage{array}

\usepackage[caption=false]{subfig}

\usepackage{diagbox}
\usepackage{graphicx}
\usepackage{epstopdf}

\usepackage{amsopn}
\usepackage{amsmath}
\usepackage{multirow}
\usepackage{color}
\usepackage{float}

\begin{document}

\markboth{W. Jiang, Z. Zhang, Z. Zhou}{A regularized model for wetting/dewetting problems}

%
\catchline{}{}{}{}{}
%

\title{A regularized model for wetting/dewetting problems: \\asymptotic analysis and $\Gamma$-convergence  }

\author{Wei Jiang}

\address{School of Mathematics and Statistics, Hubei Key Laboratory of Computational Science, Wuhan University, Wuhan 430072, China. \\
	jiangwei1007@whu.edu.cn}

\author{Zhen Zhang \thanks{Corresponding authors.} }

\address{Department of Mathematics, Guangdong Provincial Key Laboratory of Computational Science and Material Design, International Center for Mathematics, National Center for Applied Mathematics (Shenzhen), Southern University of Science and Technology,
	Shenzhen 518055, China. \\
	zhangz@sustech.edu.cn}

\author{Zeyu Zhou $^*$}

\address{Department of Mathematics, Southern University of Science and Technology,\\
	Shenzhen 518055, China. \\
	zhouzy2021@mail.sustech.edu.cn}

\maketitle

\begin{history}
\received{(Day Month Year)}
\revised{(Day Month Year)}
\comby{(xxxxxxxxxx)}
\end{history}

\begin{abstract}
By introducing height dependency in the surface energy density, we propose a novel regularized variational model to simulate wetting/dewetting problems.
The regularized model leads to the appearance of a precursor layer which covers the bare substrate, with the precursor height depending on the regularization parameter $\varepsilon$. The new model enjoys lots of advantages in analysis and simulations. With the help of the precursor layer, the regularized model is naturally extended to a larger domain than that of the classical sharp-interface model, and thus can be solved in a fixed domain.
There is no need to explicitly track the contact line motion, and difficulties arising from free boundary problems can be avoided. In addition, topological change events can be automatically captured.
Under some mild and physically meaningful conditions, we show the positivity-preserving property of the minimizers of the new model.
By using asymptotic analysis and $\Gamma$-convergence, we investigate the convergence relations between the new regularized model and the classical sharp-interface model.
Finally, numerical results are provided to validate our theoretical analysis, as well as the accuracy and efficiency of the new regularized model.
\end{abstract}



\section{Introduction}\label{intro}

Wetting/dewetting phenomena are ubiquitous in nature and technology.\cite{deGennes85,thompson2012} In general, when a solid/liquid drop is placed on a substrate which it wets, it may spread out to form a film; conversely, when a solid/liquid film previously covers a non-wetted substrate, it may dewet or agglomerate to form isolated drops. The kinetic process of wetting or dewetting is determined by the presence of the three-phase contact line which separates ``wet'' regions from those that are either dry or covered by a microscopic film. Nowadays, wetting and dewetting have been considered as both fundamental modes of motion of solid/liquid films on a substrate, and they are extremely important for applications in materials, chemistry, biology, engineering and industry.

According to the film's material state, {\it{wetting/dewetting phenomena can be classified into two categories: liquid-state wetting/dewetting and solid-state wetting/dewetting.}} The research about liquid-state wetting/dewetting can be dated back to the pioneering work of Young and Laplace two hundreds years ago.\cite{Young1805} They discovered that one of the primary parameters that characterize wetting/dewetting of liquids is the equilibrium or static contact angle, which is determined by the well-known Young's equation
\begin{equation*}
\cos\theta_e=\frac{\gamma_{_{VS}}-\gamma_{_{LS}}}{\gamma_{_{LV}}},
\end{equation*} 	
where $\gamma_{_{VS}}$, $\gamma_{_{LS}}$ and $\gamma_{_{LV}}$ are the surface tension (i.e., surface energy density) of the vapor-substrate interface, the liquid-substrate interface and the liquid-vapor interface, respectively. In the past few decades, liquid-state wetting/dewetting has been gradually becoming one of the core problems in fluid mechanics and applied mathematics.\cite{QianWangSheng2006,Bonn2009wetting,Xu2011analysis,Andreotti2020statics} For more details about liquid-state wetting/dewetting, readers may refer to the review articles (e.g., Refs. \refcite{deGennes85,Bonn2009wetting}). Very recently, wetting/dewetting on deformable substrates have attracted much attention due many interesting effects such as the formation of cusps and their potential applications in self-transport processes.\cite{marchand2012contact,style2013universal,zhang20static,Bradley19,zhang22variational}

On the other hand, solid-state wetting/dewetting is a fundamental phenomenon in materials science, and has been widely observed in a large number of experimental systems (e.g., see the recent review papers~\refcite{thompson2012,Leroy16}), such as Ni films on MgO substrates and SOI (i.e., Si films on amorphous SiO$_2$ substrates). Usually, solid thin films deposited on a substrate are usually metastable or unstable in the as-deposited state and can exhibit complex morphological evolutions when the temperature is well below the melting point of the film material. Because the film is still in solid-state, this phenomenon is called the solid-state wetting/dewetting, and it has attracted increasing attentions in recent years not only because of its arising scientific problems related to modelling and simulations~\cite{jiang2012,Jiang16,bao2017,Naffouti17,Tripathi2018,zhao2020parametric,Boccardo2022}, but also due to its wide applications in industrial technologies, used in solar cells, optical and magnetic devices, sensor devices, catalyzing the growth of carbon nanotubes, semiconductor nanowires and so on (cf. Refs. \refcite{armelao2006,danielson2006,randolph2007} and \refcite{volker2009}).

\begin{figure}[tbhp]
	\centering
	\includegraphics[width=12.0cm]{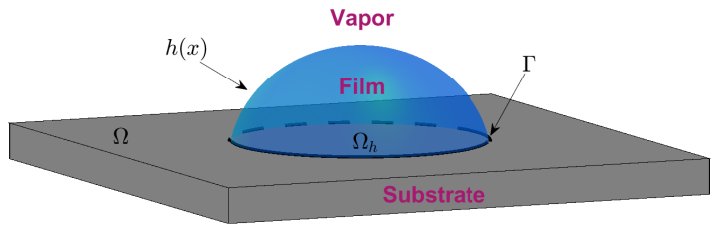}
	\caption{\label{fig:dro}{An island film (which is in liquid-state or solid-state) sitting on a flat, rigid substrate with three interfaces,  i.e.,  film/vapor (FV),  film/substrate  (FS), and vapor/substrate (VS) interfaces, where $\Gamma$} is the contact line.}
\end{figure}

Both liquid-state and solid-state wetting/dewetting problems can be modelled in a simplified situation.
As shown in Fig.~\ref{fig:dro}, we consider a film (which is in liquid-state or solid-state) standing on a flat, rigid substrate, and assume that the film-vapor interface is parameterized by a height function $h=h(x)$ for $x\in\Omega\subset R^n$, $n=1, 2$, where $\Omega$ is an open and bounded domain with its boundary $\partial\Omega$ being Lipschitz continuous. Physically in wetting/dewetting problems, $h$ is usually supported on an $h$-dependent subset $\Omega_h\subset\Omega$, with its boundary $\Gamma:=\partial\Omega_h$ being the contact line, i.e., a triple line where the film, substrate and vapor phases meet.
Following the thermodynamic principles laid out by Gibbs,\cite{Gibbs1878} the equilibrium shapes of wetting/dewetting problems can be determined by minimizing the total interfacial energy functional of the system, i.e., finding the equilibrium profile of $h\in W^{1,1}(\Omega_h)$ in order to minimize the following functional
\begin{equation}
 E=\gamma_{_{VS}}|\Omega\setminus\Omega_h|+\gamma_{_{FS}}|\Omega_h|+\int_{\Omega_h}\gamma_{_{FV}}\sqrt{1+|\nabla h|^2}\mathrm{d}x,
 \label{eqn:func}
\end{equation}
under the volume constraint $\int_{\Omega_h}h(x)\mathrm{d}x=V$, where $V$ is a prescribed positive constant, $\gamma_{_{VS}}$, $\gamma_{_{FS}}$ and $\gamma_{_{FV}}$ are respectively the surface energy densities of the vapor/substrate, film/substrate and film/vapor interfaces. It is noted here that, although the above model is relatively simple and we only focus on the surface energy effect (e.g., for liquids, we ignore the hydrodynamical effect; for solids, we ignore
the elastic effect, alloy, grain boundary and others), it includes the most fundamental and important features of wetting/dewetting problems.
{\it {It can be regarded as a basic model for studying wetting/dewetting problems, and many more complicated models can be built on it.}}

For liquid-state problems, $\gamma_{_{VS}}$, $\gamma_{_{FS}}$ and $\gamma_{_{FV}}$ are often three material constants, which determines the equilibrium contact angle by the Young's equation, i.e., $\cos\theta_e=({\gamma_{_{VS}}-\gamma_{_{FS}}})/{\gamma_{_{FV}}}$. If $\theta_e \in (0, \pi)$, then the equilibrium shape is a part of sphere which is
truncated by the flat substrate at a contact angle of $\theta_e$. For solid-state problems, $\gamma_{_{VS}}$, $\gamma_{_{FS}}$ are often two constants, but due to the lattice orientational difference of solids, $\gamma_{_{FV}}$ can be a function of the surface normal of the film-vapor interface. This is called the anisotropic surface energy.\cite{bao2017,wang2015sharp,bao2017parametric} To predict equilibrium shapes without considering the substrate energy, Wulff proposed a geometrical approach called ``Wulff construction''.\cite{wulff1901} This method works well not only for the isotropic case but also for anisotropic cases, and has been rigorously proved by Taylor and Fonseca.\cite{taylor1974,fonseca1991A} Furthermore, based on the idea of the Wulff construction, the ``Winterbottom construction'' was proposed to predict equilibrium shapes of solid-state wetting/dewetting with the substrate energy being considered.\cite{winterbottom1967,bao2017} Interestingly, motivated by many experiments which indicate that there exist multiple stable equilibrium shapes in strongly anisotropic surface energy cases, Jiang and his collaborators proposed a generalized Winterbottom construction to identify all possible equilibrium states of solid-state wetting/dewetting in Refs.~\refcite{Jiang16,bao2017}. Recently, Piovano and Vel{\v{c}}i{\'c} have showed that
the above continuum model \eqref{eqn:func} for solid-state wetting/dewetting is $\Gamma$-converged by an atomistic model taking into consideration atomic interactions of solid particles both among themselves and with the fixed substrate atoms.\cite{Piovano2022}

Furthermore, the kinetic evolution problems of wetting/dewetting can be mathematically regarded as gradient flows of the above energy functional \eqref{eqn:func} in some suitable metric spaces. For liquid-state wetting/dewetting, one can use the $L^2$-gradient flow with volume constraints or the Hele-Shaw flow,\cite{Alikakos1994,Chen2011mass} which are related with the mass-conserving Allen-Cahn equation\cite{Turco2009wetting,Chen2011mass} or the conventional Cahn-Hilliard equation\cite{Chen2014analysis} respectively in the phase-field framework; for solid-state wetting/dewetting, one can usually use the $H^{-1}$-gradient flow, which corresponds to the surface diffusion flow.~\cite{Cahn94,jiang2012,bao2017,bao2017parametric,jiang2020} Nevertheless, because the above energy functional \eqref{eqn:func} depends on an unknown film/vapor interface, which intersects with the substrate along a contact line $\Gamma$, the governing equation for describing the motion of the interface will include contact line migration. Therefore, the kinetic problems belong to a type of open curve/surface evolution problems under geometric flows together with contact line migration, and this type of free boundary problems has posed a considerable challenge to researchers in materials science, applied mathematics and scientific computing.~\cite{wong2000,du2010,bao2017parametric,zhao2020parametric,jiang2020}

In addition, the pinch-off events often take place during the kinetic evolution of wetting/dewetting, i.e., a continuous film is split into two or more parts.
In general, sharp-interface approaches can not deal with the topological change events, especially for three-dimensional cases.\cite{du2010,jiang2020}
In order to tackle the difficulty, phase-field models are proposed in the literature for simulating wetting/dewetting problems.~\cite{jiang2012,dziwnik2017}
The phase-field models introduce an artificial scalar function (i.e., the phase-field function) to describe the evolution of the interface. The interface/surface is represented by the zero-level set of the phase function. Recently, a phase field model by using the Cahn-Hilliard equation with degenerate mobility and nonlinear boundary conditions along the substrate has been proposed for simulating the solid-state dewetting with isotropic surface energy~\cite{jiang2012,Huang2019}, and this method was recently extended to the weakly anisotropic surface energy case.\cite{dziwnik2017} The advantages of these phase-field models are their ability to automatically dealing with topological changes, and that they are easily extended to three dimensional cases. However, the phase-field model suffers from its lower computational efficiency than that of sharp-interface models, since it is one dimension higher in space. Furthermore, in order to approximate the correct geometric flow (e.g., surface diffusion), the commonly used Cahn-Hilliard equation should include the higher-order degenerate mobility~\cite{cahn1996,dai2014,lee2016,Bretin2022approximation}, which will pose considerable difficulties in developing efficient and accurate numerical schemes for solving it.

Since the existing models (including sharp-interface and phase-field models) have more or less deficiencies and limitations, this paper aims to propose a new variational regularized model for solving wetting/dewetting problems. The new regularized model enjoys all advantages of the previous models: (1) it only needs to be solved in a fixed domain, no matter what kind of problems (i.e., equilibrium or kinetic problems) we focus on; (2) it does not need to explicitly handle the contact line motion, so it can automatically capture the topological change events; (3) it is one dimension lower than the phase-field model, so its computational cost is relatively smaller. The key idea is that we relax the above energy functional \eqref{eqn:func} by introducing a relaxed surface energy density $\gamma^\varepsilon(h)$ (see section 2), where $h$ is a height function of the film, and $\varepsilon$ is a small regularization parameter. By our theoretical analysis, we show that the new relaxed surface energy density will lead to the effect that the bare substrate is always covered by a precursor layer whose thickness depends on the small parameter $\varepsilon$. This brings lots of advantages in analysis and simulations. By extending the energy functional \eqref{eqn:func} into an integral over a fixed domain, we then transform the free boundary problem into a fixed domain problem. As a first step, we will rigorously show that the solutions of the regularized model is positivity-preserving. We also prove that the regularized functional $\Gamma$-converges to the original energy functional \eqref{eqn:func}. Furthermore, by using asymptotic analysis, we show that the new model can perfectly recover the well-known Young's equation and the solution of this model asymptotically converges to that of the original model (including the film profile far from the substrate, the height of the precursor, and the contact line region) with satisfactory rate.

The rest of the paper is organized as follows. In section 2, we first propose a new variational regularized model and establish the positivity-preserving property of its minimizers in the sense of classical solutions. Then, we show some properties of the asymptotic convergence in section 3 by using the formal matched asymptotic expansion for the new regularized model. In section 4, we present a mathematical proof for the $\Gamma$-convergence of the regularized model to the classical sharp-interface model. Finally, some numerical results are shown in section 5 to validate the theoretical analysis of the new model.

\section{A new regularized model and the positivity-preserving property of its minimizers}\label{sec:odes}
%
As we discuss before, the minimization problems (including equilibrium and kinetic problems) with respect to the functional \eqref{eqn:func} belong to a type of free boundary problems, which will bring much difficulties in both analysis and numerical simulations. To tackle these difficulties, we extend the support region of the height function $h(x)$ to the whole domain $\Omega$ by introducing a precursor thin film layer outside the part of the solid/liquid film (shown in Fig. \ref{fig:precursor-droplet}). The precursor is also represented by $h(x)$, and its height is very small (i.e., $0 < h(x)\ll1 $). An illustration about the difference between the original sharp-interface model and the new regularized model is depicted in Fig. \ref{fig:precursor-droplet}.
\begin{figure}[tbhp]
	\centering
	\includegraphics[width=12.5cm]{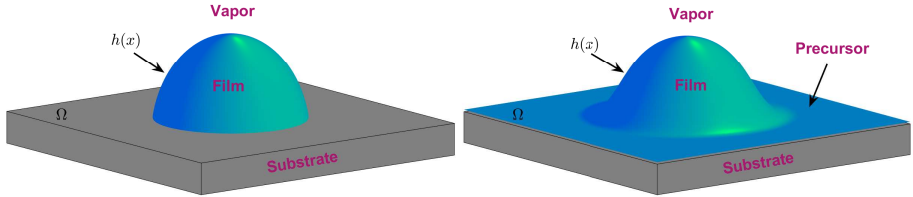}
	\caption{\label{fig:precursor-droplet} A schematic illustration about the difference between the original sharp-interface model (as shown left) without a precursor and the new regularized model (as shown right) with a precursor, where the thickness of the precursor is at $\mathcal{O}(\varepsilon^2)$ (shown in Section 3).}
\end{figure}

Based on the key idea, we introduce a regularized surface energy density $\gamma^\varepsilon(h)$ as follows
\begin{equation}
	\gamma^\varepsilon(h)=\gamma_{_{FV}}+(\gamma_{_{FV}}-\gamma_{_{VS}}+\gamma_{_{FS}})g(\frac{h}{\varepsilon})=\gamma_{_{FV}}-Sg(\frac{h}{\varepsilon}),
\label{eqn:reg}
\end{equation}
where $\varepsilon$ is a regularized parameter, i.e., $0<\varepsilon\ll 1$, and $S=\gamma_{_{VS}}-\gamma_{_{FS}}-\gamma_{_{FV}}$ is the spreading parameter.
{\it {For simplicity, we always consider the isotropic system, i.e., $\gamma_{_{VS}}$, $\gamma_{_{FS}}$ and $\gamma_{_{FV}}$ are three positive constants
in this paper.}}

In \eqref{eqn:reg}, $g(z)$ is an interpolation function between $-1$ and $0$ when $z>0$, and
it is easy to see that the regularized density function $\gamma^\varepsilon(h)$ interpolates between $\gamma_{_{FV}}$ and $\gamma_{_{VS}}-\gamma_{_{FS}}$.
More precisely, when the film thickness approaches zero (i.e., close to a bare substrate), $g(z)$ goes to $-1$ and it is expected that $\gamma^\varepsilon(h)$ converges to the
energy density $\gamma_{_{VS}}-\gamma_{_{FS}}$; when the film thickness is very large, $g(z)$ goes to $0$ and it is expected that $\gamma^\varepsilon(h)$ converges to the film/vapor surface energy density  $\gamma_{_{FV}}$. It is noted that a similar version of the regularized density $\gamma^\varepsilon(h)$ was proposed for studying the epitaxial solid film growth and triple line kinetics,\cite{Chiu1994,Tripathi2018} and similar thickness-dependent surface energy densities can also be found in fluid mechanics for modeling the wetting/dewetting of thin liquid films.\cite{Peschka2019,Wu2004slope}

In addition, we also assume that$\lim\limits_{z\to-\infty}g(z)=+\infty$ in order to prevent $h$ from taking negative values. Define Young's angle $\theta_e$ by $\cos\theta_e=(\gamma_{_{VS}}-\gamma_{_{FS}})/\gamma_{_{FV}}$. Throughout this paper, we will only consider the partial wetting/dewetting regime with $\theta_e \in (0,\frac{\pi}{2})$ due to the graph representation of film thickness. {\it {Thus, we always assume that $S<0$ and $\gamma_{_{VS}}-\gamma_{_{FS}}>0$ in this paper}}.

In summary, we assume that the interpolation function $g(z)$ satisfies the following conditions throughout this paper:
\smallskip
\begin{itemize}
\item[(1)] $g(z) \in C^1(\mathbb{R})$ with $g(0)=-1$, $\lim\limits_{z\to+\infty}g(z)=0$;

\item[(2)] $g(z)$ is strictly increasing if $z>0$, and strictly decreasing if $z<0$.
\end{itemize}
We note that the above two conditions will be used in the proof of the positivity-preserving property.
Furthermore, in order to perform asymptotic analysis and $\Gamma$-convergence about the new model,
we still need two additional conditions:
\begin{itemize}
\item[(3)] $g(z) \in C^\infty(\mathbb{R})$ and $ \lim\limits_{z\to+\infty}zg'(z)=0 $;

\item[($3^\prime$)] $\lim\limits_{z\to-\infty}g(z)=+\infty$, and $ g(z)\leqslant g(-z) $ if $ z>0 $.
\end{itemize}
Here, ($3$) will be used in the asymptotic analysis, and ($3^\prime$) used in the proof of $\Gamma$-convergence.
An example~\cite{Tripathi2018} of the function $g(z)$  which satisfies the above conditions is taken as $g(z)=e^{-z}-2e^{-\frac{z}2}$,
as shown in Fig. \ref{fig:profile_g}.
\begin{figure}[tbhp]
	\centering
	\includegraphics[width=8.0cm]{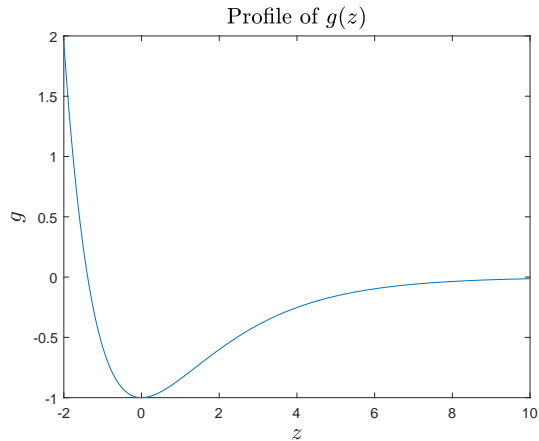}
	\caption{\label{fig:profile_g}An example of the interpolation function $g(z)$ defined in \eqref{eqn:reg}.}
\end{figure}

Using the regularized density $\gamma^\varepsilon(h)$ in \eqref{eqn:reg}, we can define the regularized interfacial energy functional of the system
in the fixed domain $\Omega$ as
\begin{equation}
	F^\varepsilon(h)=\int_{\Omega}\gamma^\varepsilon(h)\sqrt{1+|\nabla h|^2}\mathrm{d}x.
\label{eqn:regfun}
\end{equation}
 The equilibrium state is obtained by minimizing $F^\varepsilon(h)$ subject to the volume constraint
\begin{equation}\label{eq:vol}
G(h)=\int_{\Omega}h(x)\mathrm{d}x=V.
\end{equation}
From the Euler-Lagrange equation, we have
\begin{equation}\label{eq:equilib}
	\frac{\delta}{\delta h}(F^\varepsilon(h)-\lambda G(h))=\frac{\delta F^\varepsilon}{\delta h}-\lambda=0,
\end{equation}
where $\lambda$ is the Lagrange multiplier, and
\begin{align*}
	\frac{\delta F^\varepsilon}{\delta h}=&\frac{\mathrm{d}\gamma^\varepsilon}{\mathrm{d}h}\sqrt{1+|\nabla h|^2}-\nabla\cdot\Big(\frac{\gamma^\varepsilon(h)\nabla h}{\sqrt{1+|\nabla h|^2}}\Big)\\
	=&\frac{\mathrm{d}\gamma^\varepsilon}{\mathrm{d}h}\frac1{\sqrt{1+|\nabla h|^2}}-\gamma^\varepsilon(h)\nabla\cdot\Big(\frac{\nabla h}{\sqrt{1+|\nabla h|^2}}\Big).
\end{align*}
The equilibrium profile of $h$ can be obtained by solving \eqref{eq:equilib} together with the Neumann boundary condition:
\begin{equation}\label{eq:Neumann-bd}
	\frac{\partial h}{\partial\boldsymbol{\nu}}=0,\qquad\mbox{on}\quad \partial \Omega,
\end{equation}
where $ \boldsymbol{\nu} $ is outward normal vector and the constant $\lambda$ can be determined via the volume constraint.

First, we show that the equation \eqref{eq:equilib} with the boundary condition \eqref{eq:Neumann-bd} and the constraint \eqref{eq:vol} always admits
a positive solution in the classical sense, which means that the precursor layer always appears. To prove this, we need the interior sphere condition given below which is necessary when we use the Hopf's lemma.

\begin{definition}
	The domain $\Omega$ is said to satisfy interior sphere condition at $ x_0 $ if there exists an open ball $ B \subset \Omega$ with $ x_0\in\partial B $. For example, if $\partial\Omega$ is $ C^2 $-smooth, then $\Omega$ satisfies interior sphere condition at every point on $\partial\Omega$.
\end{definition}

Now, we can present the positivity-preserving property of the solution.

\begin{theorem}\label{positivity}
	Assume that $\Omega\in\mathbb{R}^n$ is open, $ n=1 \text{ or } 2$, and it satisfies interior sphere condition at any point on $ \partial\Omega $. Also assume that the interpolation function $g(z)$ satisfies the conditions (1) and (2), and the Young's angle $\theta_e\in(0,\frac{\pi}{2})$. Let the admissible set be $ \mathcal{A}:=\{u\in C^2(\Omega)\cap C^0(\bar{\Omega}):\int_{\Omega}u\mathrm{d}x=V>0\} $. For any fixed $\varepsilon>0$, if $ h\in \mathcal{A} $ satisfies
	\[
	F^\varepsilon(h)=\min_{u\in\mathcal{A}}F^\varepsilon(u),
	\]
	then $ h>0 $ in $ \bar{\Omega}$ and the corresponding Lagrange multiplier $ \lambda>0 $.
\end{theorem}
\begin{proof}
	We only consider the case $ n=2 $. For the case $n=1$, it is similar.
	
	Let us first define a quasi-linear strictly elliptic operator as
	\[
	\mathcal{L}h:=(1+h_y^2)h_{xx}-2h_xh_yh_{xy}+(1+h_x^2)h_{yy},
	\]
	which is widely used in the minimal surface equation.

	The Euler-Lagrange equation \eqref{eq:equilib} can be recast as
	\begin{equation}\label{pde_equ}
		\gamma^\varepsilon(h)\mathcal{L}h=-\bigg(\frac{1}{\varepsilon}Sg'(\frac{h}{\varepsilon})(1+|\nabla h|^2)+\lambda(1+|\nabla h|^2)^{\frac{3}{2}}\bigg),
	\end{equation}
	associated with the Neumann boundary condition $ \frac{\partial h}{\partial\boldsymbol{\nu}}=0$  on $\partial\Omega  $.
	
	\textbf{Case I}: $\lambda>0$. We will argue by contradiction that $h>0$ in $\bar{\Omega}$. Assume $ \min_{x\in\bar{\Omega}}h(x)\leqslant0 $, and let $ h(x) $ attain its minimum at $ x_0 \in \bar{\Omega}$. We separate our arguments into the following two cases:
	
	(i) If $ x_0\in\Omega $, we have $ h(x_0)\leqslant0 $, $ h_x(x_0)=0 $, $ h_y(x_0)=0 $ and its Hessian matrix $ \nabla^2 h(x_0) $ is non-negative definite.  Then at $ x_0 $, \eqref{pde_equ} becomes:
	\[
	\gamma^\varepsilon(h(x_0))\Delta h(x_0)=-\bigg(\frac{1}{\varepsilon}Sg'(\frac{h(x_0)}{\varepsilon})+\lambda\bigg).
	\]
	By condition (2), $g(z)$  is decreasing when $z\leqslant0 $, which implies $ g'(\frac{h(x_0)}{\varepsilon}) \leqslant0$. Since $ \gamma^\varepsilon(h)\geqslant\gamma_{_{FV}}+S=\gamma_{_{VS}}-\gamma_{_{FS}}>0 $ and $ S<0 $, we know $ \text{Tr}(\nabla^2h(x_0))=\Delta h(x_0)<0 $ which leads to a contradiction. 
	
	(ii) If $ x_0\in\partial\Omega $, then for any $ \delta>0 $, by the continuity of $ h $ there exists an open neighborhood $ N $ of $ x_0 $ such that $ h(x)\leqslant\delta $ for any $ x\in N\cap\Omega $. We will show that if $\delta$ is small enough, $ \mathcal{L}h(x)\leqslant0 $ for any $ x\in N\cap\Omega $.
	
	When  $ h(x)\leqslant0 $ for some $ x\in N\cap\Omega $, similar as in the previous analysis, we can get $ \mathcal{L}h(x)\leqslant0 $.
		
	When  $ 0<h(x)\leqslant\delta $ for some $ x\in N\cap\Omega $, 
	we have
	$ 0<\gamma_{_{VS}}-\gamma_{_{FS}}\leqslant\gamma^\varepsilon(h(x))\leqslant \gamma_{_{FV}} $.
	Because $ g'(0)=0 $, $g'(z)>0$ for $z>0$, and $ g'  $ is
	continuous, if $\delta$ is small enough, we know $ -\lambda<\frac1{\varepsilon}Sg'(\frac{h(x)}{\varepsilon})<0 $. Combining $ (1+|\nabla h|^2)^{\frac{1}{2}}\geqslant 1 $ and that $\lambda>0$ is a constant, we get
	\begin{equation*}
		\mathcal{L}h(x)\leqslant\frac{-\bigg(\frac{1}{\varepsilon}Sg'(\frac{h(x)}{\varepsilon})+\lambda\bigg)(1+|\nabla h(x)|^2)}{\gamma^\varepsilon(h(x))}<0.
	\end{equation*}

	Now that $ \mathcal{L}h(x)\leqslant0 $ for any $ x\in N\cap\Omega $, in view of $ x_0 $ is the minimum point of $ h $, by the strict ellipticity of $\mathcal{L}$ and Hopf's lemma, we have $ \frac{\partial h}{\partial\boldsymbol{\nu}}(x_0)<0 $, which contradicts to the Neumann boundary condition.
	
	Therefore, $ h>0 $ in $\bar{\Omega}$.
	
	\textbf{Case II}: $\lambda\leqslant0$. We will show that this case is impossible. This can be done by first proving that $ \max_{x\in\bar{\Omega}}h(x)\leqslant 0 $. Assume $ \max_{x\in\bar{\Omega}}h(x)>0 $ and the maximum is attained at $ x_0 \in\bar{\Omega}$. Similar (and simpler) arguments as in Case I will lead to contradictions.

	
	
	
	
	Hence, $ \max_{x\in\bar{\Omega}}h(x)\leqslant0 $, which implies $ \int_{\Omega}h\mathrm{d}x\leqslant0 $. This contradicts to the assumption that $h\in \mathcal{A}$. Therefore, Case II is impossible.
	
	In summary, $ h>0 $ in $\bar{\Omega}$ and $ \lambda>0 $.
\end{proof}

\section{Asymptotic analysis for the equilibrium state}
In this section, we always assume that $g(z)$ satisfies the conditions (1), (2) and (3).
As $\varepsilon\rightarrow 0$, it is indicated from the numerical simulations that there exists a transition layer near the contact line $\Gamma$ where the derivatives of $h(x)$ changes dramatically. We can roughly determine the contact line $\Gamma$ by the constant mean curvature film surface together with the volume constraint.
This indicates that there is a singular behavior of the equilibrium profile $h(x)$ which can be analyzed by using matched asymptotic analysis.

To do so, we first define a signed distance function $\phi(x)=\text{dist}(x,\Gamma)$ between $x$ and the contact line $\Gamma$. We assume that when the point $x$ belongs to the film region, the value of $\phi(x)$ is positive.

\textbf{Outer expansion:}

Let the outer solutions be expanded as
\begin{equation*}
	h_{out}^i=h_0^i+\varepsilon h_1^i+\varepsilon^2 h_2^i+\cdots,\qquad \lambda_{out}^i=\lambda_0^i+\varepsilon \lambda_1^i + \varepsilon^2 \lambda_2^i+\cdots,
\end{equation*}
where the superscript $i=f$ corresponds to the film region and $i=p$ stands for the precursor region. Then
\[
\gamma^\varepsilon(h_{out}^i)=\gamma_{_{FV}}-S\cdot g\Big(\frac{h_{out}^i}{\varepsilon}\Big),\quad\mbox{and}\quad \frac{\mathrm{d}\gamma^\varepsilon}{\mathrm{d}h}(h_{out}^i)=-\frac1{\varepsilon}S\cdot g'\Big(\frac{h_{out}^i}{\varepsilon}\Big).
\]
In the film region ($\phi(x)\gg \varepsilon$), the height is positive with $h_0^f>0$.  Since $ \lim\limits_{z\to+\infty}g(z)=0 $, we have
\[
\Big|g\Big(\frac{h_{out}^f}{\varepsilon}\Big)\Big|\sim \Big|g\Big(\frac{h_{0}^f}{\varepsilon}\Big)\Big|\ll1,\text{ as }\varepsilon\to0.
\]
Because $ \lim\limits_{z\to+\infty}zg'(z)=0 $,
\[
\frac{1}{\varepsilon}g'\Big(\frac{h_{out}^f}{\varepsilon}\Big)\sim \frac{1}{h_{0}^f}\frac{h_{0}^f}{\varepsilon}g'\Big(\frac{h_{0}^f}{\varepsilon}\Big)\ll1,\text{ as }\varepsilon\to0.
\]
Hence, the leading term of $ \gamma^\varepsilon(h_{out}^f) $ is $ \gamma_{_{FV}} $ and $ \frac{\mathrm{d}\gamma^\varepsilon}{\mathrm{d}h}(h^f_{out})\ll1 $. From the governing equation \eqref{eq:equilib}, we have for the leading order terms in $\mathcal{O}(1)$ that
\begin{align}\label{eq:const-curv}
	-\gamma_{_{FV}}\nabla\cdot\Big(\frac{\nabla h_{0}^f}{\sqrt{1+|\nabla h_{0}^f|^2}}\Big)=\lambda_0^f.
\end{align}
This is exactly the constant mean curvature condition which implies a spherical cap shape of the film surface. From the results in Theorem \ref{positivity} and the definition of asymptotic series, we know that $ \lambda_0^f>0 $ when $\varepsilon$ is small enough. This implies that the shape of the film surface is concave which is consistent with numerical simulations.

In the precursor region ($-\phi(x)\gg \varepsilon$), it is suggested that  $ h^p_{out}\ll1 $ from the definition of the precursor region. Hence, $h_0^p=0$.

Then $\gamma^\varepsilon(h^p_{out})=\gamma_{_{FV}}-Sg(h_1^p)+\mathcal{O}(\varepsilon)$ and $\frac{\mathrm{d}\gamma^\varepsilon}{\mathrm{d}h}(h^p_{out})=-\frac{S}{\varepsilon}g'(h_1^p)-Sg''(h_1^p)h_2^p+\mathcal{O}(\varepsilon)$. Substituting these equations into the governing equation \eqref{eq:equilib}, for the leading order terms in $\mathcal{O}(\frac1{\varepsilon})$, we obtain that
\begin{equation*}
	g'(h_1^p)=0.
\end{equation*}
Since $ h_1^p $ is finite and $ g'(z) $ only has one zero point at $ z=0 $, we obtain $h_1^p=0$. The first order terms in $\mathcal{O}(1)$ leads to
\begin{equation*}
	-Sg''(h_1^p)h_2^p=\lambda_0^p,
\end{equation*}
which implies that $h_2^p=-\frac{\lambda_0^p}{Sg''(0)}>0$ since $ g(z)$ attains its minimum at $ z=0$ strictly. It should be noted that $\lambda_0^p$ and $\lambda_0^f$ should be the same due to the matching condition explained below. We shall denote it by $\lambda_0$ without any superscript. In summary, the outer solution in the precursor region is given by
\begin{equation}\label{eq:height}
	h^p_{out}=-\frac{\lambda_0}{Sg''(0)}\varepsilon^2+\mathcal{O}(\varepsilon^3).
\end{equation}

\textbf{Inner expansion: }

By the balance of dominant terms, it is easy to observe that the boundary layer thickness is $\mathcal{O}(\varepsilon)$. It is helpful to introduce a rescaled inner variable $ \xi:=\frac{\phi(x)}{\varepsilon} $ along the normal direction to $\Gamma$. We consider the decomposition in a local coordinate system near $\Gamma$, i.e.
\[
x(s,\xi;\varepsilon)=x(s;\varepsilon)+\varepsilon\xi\boldsymbol{\nu}(s;\varepsilon),
\]
where  $ x(s;\varepsilon) $ is a parametrization of $\Gamma$, $ s $ is its arc length parameter and $ \boldsymbol{\nu}$ is the outward unit normal vector.

The gradient operator and the Laplace operator can be recast in this local coordinate system as (cf. Ref. \refcite{dziwnik2017}):
\begin{align*}
	\nabla&=\varepsilon^{-1}\boldsymbol{\nu}\partial_\xi+\frac{1}{1+\varepsilon\xi\kappa}\textbf{t}\partial_s,\\
	\Delta&=\frac{1}{\varepsilon^{2}} \frac{1}{1+\varepsilon \xi \kappa} \partial_{\xi}\left((1+\varepsilon \xi \kappa) \partial_{\xi} \right)+\frac{1}{1+\varepsilon \xi \kappa}\partial_s\left(\frac{1}{1+\varepsilon \xi \kappa} \partial_{s} \right),
\end{align*}
where $\mathbf{t}$ is the tangent vector along $\Gamma$ and $\kappa$ is the curvature of $\Gamma$ defined through $\partial_s\boldsymbol{\nu}=\kappa\textbf{t}$.

We then rescale $ h $ as $H(s,\xi)=\frac{h(x)}\varepsilon$.
Direct calculations lead to
\begin{align*}
	&\nabla h=\boldsymbol{\nu}\cdot \partial_\xi H+\varepsilon\textbf{t}\cdot\partial_sH+\mathcal{O}(\varepsilon^2),\\
	&\Delta h=\frac{1}{\varepsilon} \partial_{\xi\xi} H+\kappa \partial_\xi H-\varepsilon\kappa^2\xi \partial_\xi H+\varepsilon\partial_{ss}H+\mathcal{O}(\varepsilon^2).
\end{align*}
In addition,  if we denote $\tilde{\gamma}(H):=\gamma_{_{FV}}-S(e^{-H}-2e^{-\frac12H})$, then $\gamma^\varepsilon(h)=\tilde{\gamma}(H)$.

Let the inner asymptotic expansion be
\begin{equation*}
	H(s,\xi)=H_0(s,\xi)+\varepsilon H_1(s,\xi)+\cdots,\qquad \lambda_{in}=\Lambda_0+\varepsilon \Lambda_1 + \varepsilon^2 \Lambda_2+\cdots.
\end{equation*}
Substituting them into the governing equation \eqref{eq:equilib}, for the leading order terms in $\mathcal{O}(\frac1{\varepsilon})$, we obtain
\begin{equation*}
	\tilde{\gamma}'(H_0)\frac1{\sqrt{1+(\partial_\xi H_0)^2}}-\tilde{\gamma}(H_0)\frac{\partial_{\xi\xi} H_0}{(1+(\partial_\xi H_0)^2)^{\frac{3}{2}}}=0.
\end{equation*}
Multiplying this equation by $ \partial_\xi H_0 $, we can integrate this equation once and obtain
\begin{equation}\label{eq:in-sol}
	\tilde{\gamma}(H_0)^2=C_0(1+(\partial_\xi H_0)^2),
\end{equation}
with some constant $C_0>0$. Using the matching conditions that
\begin{align*}
	&\lim\limits_{\phi(x)\rightarrow 0^-}h_{out}^p(x)=\lim\limits_{\xi\rightarrow-\infty}\varepsilon H,\quad \lim\limits_{\phi(x)\rightarrow 0^-}\boldsymbol{\nu}\cdot\nabla h_{out}^p(x)=\lim\limits_{\xi\rightarrow-\infty} \partial_\xi H,\\
	&\lim\limits_{\phi(x)\rightarrow 0^-}\lambda_{out}^p=\lim\limits_{\xi\rightarrow-\infty}\lambda_{in},
\end{align*}
we have $\lambda_i^p=\Lambda_i$ for $i=0,1,\ldots$, and $\lim\limits_{\xi\rightarrow-\infty}H_0=\lim\limits_{\xi\rightarrow-\infty}\partial_\xi H_0=0$. Taking $\xi\rightarrow-\infty$ in \eqref{eq:in-sol}, we obtain $C_0=(\gamma_{_{VS}}-\gamma_{_{FS}})^2$. Similarly, we have the matching conditions when $\xi\rightarrow+\infty$:
\begin{equation}\label{height_matching}
	\begin{aligned}
	&\lim\limits_{\phi(x)\rightarrow 0^+}h_{out}^f(x)=\lim\limits_{\xi\rightarrow+\infty}\varepsilon H,\quad\lim\limits_{\xi\rightarrow+\infty}\partial_\xi H_0=\lim\limits_{\phi(x)\rightarrow 0^+}\boldsymbol{\nu}\cdot\nabla_{x}h_{0}^f(x)=\tan\theta_e,
	\end{aligned}
\end{equation}
and $\lambda_i^f=\Lambda_i$ ($i=0,1,\ldots$), where $\theta_e$ is the apparent contact angle between the film surface and the substrate. Note that we focus on the case  $\theta_e<\frac{\pi}2$. The positivity of the film volume $V$ guarantees that $\theta_e>0$.  The first matching condition of \eqref{height_matching} implies that $$ \lim\limits_{\xi\rightarrow+\infty}H_0=\lim\limits_{\phi(x)\rightarrow 0^+}\frac{h_{0}^f(x)}{\varepsilon} =+\infty. $$ Taking $\xi\rightarrow+\infty$ in \eqref{eq:in-sol}, we obtain that
\begin{equation}\label{eq:Young}
	\gamma_{_{FV}}\cos\theta_e=\gamma_{_{VS}}-\gamma_{_{FS}},
\end{equation}
which recovers with the well-known Young's equation. The detailed profile of the inner solution can be solved from \eqref{eq:in-sol}:
\begin{equation}\label{eq:inner}
	H_0(s,\xi)=q^{-1}(s,\xi),\quad \mbox{with} \quad q(s,H)=\int\frac{\mathrm{d}H}{\sqrt{\Big(\frac{\tilde{\gamma}(H)}{\gamma_{_{VS}}-\gamma_{_{FS}}}\Big)^2-1}},
\end{equation}
where $ H_0(s,\xi) $ is the inverse function of $ q(s,H) $ with respect to $H$.

It is obvious that $q(s,H)$ is monotonically increasing from $0$ to $+\infty$ with respect to $H$. Thus $ H_0(s,\xi)$ is also increasing in $\xi$. The leading order profile of the inner solution is illustrated in Figure \ref{fig:inner}.

\begin{figure}[tbhp]
	\centering
	\includegraphics[width=10.0cm]{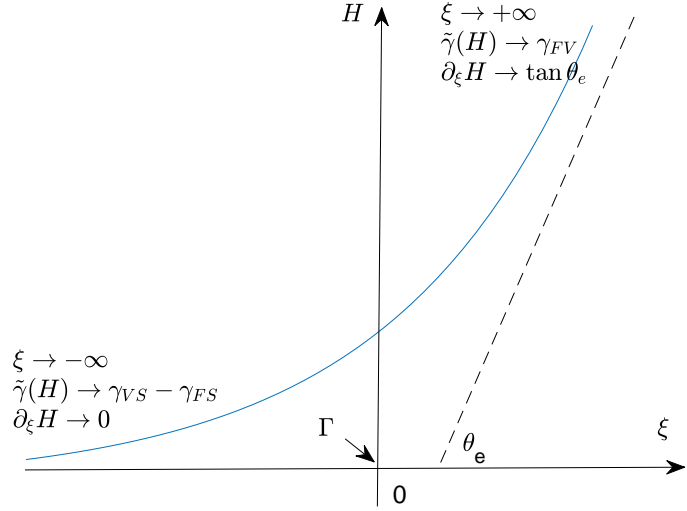}
	\caption{\label{fig:inner}An illustration of the inner solution by the asymptotic analysis.}
\end{figure}

To summarize, by using the above asymptotic analysis, we conclude that as $\varepsilon$ goes to zero, the regularized model \eqref{eq:equilib} asymptotically approaches the original sharp-interface model in the following sense:
\begin{enumerate}
	\item The Young's equation \eqref{eq:Young} is perfectly recovered in the macroscopic scale;
	\item The equilibrium shape approaches to a spherical cap, and the constant mean curvature of the spherical cap can be determined by using the volume constraint;
	\item The thickness of the precursor layer will decrease to zero at the second-order rate, i.e., $\mathcal{O}(\varepsilon^2)$.
\end{enumerate}
We note that the above conclusions will be also verified by our numerical simulations in Section 5.

\section{$ \Gamma $-convergence}

In this section, we shall prove the $\Gamma$-convergence of the proposed model to the original sharp-interface model. The convergence result and its proof rely on some preliminary knowledge about bounded variation functions, and readers may refer to \ref{bv_preliminary}.

\subsection{Energy functional and the convergence result}
Let us recall the energy density function of the proposed model:
\begin{equation*}
	\gamma^{\varepsilon}(h)=\gamma_{_{FV}}-S g\left(\frac{h}{\varepsilon}\right).
\end{equation*}
where $S=\gamma_{_{VS}}-\gamma_{_{FV}}-\gamma_{_{FS}}$.
For any positive constant $ V $, we extend the free energy functional for any $ h\in L^1(\Omega) $. Define
\begin{equation}\label{func:precursor}
	F^{\varepsilon}(h)=\begin{cases}
		\int_{\Omega} \gamma^{\varepsilon}(h)\sqrt{1+|\nabla h|^2}\mathrm{d}x, \quad\quad &\text{if } \, h\in W^{1,1}(\Omega),\ \int_{\Omega}h\mathrm{d}x=V,\\
		+\infty, &\text{otherwise in } L^1(\Omega).
	\end{cases}
\end{equation}
In contrast, the energy functional of the original sharp-interface model is recast as:
\begin{equation}\label{func:sharp}
	F(h)=\begin{cases}
		\gamma_{_{FV}}\int_{\Omega}\sqrt{1+|Dh|^2}+S\cdot\mathcal{L}^n(\mathcal{N}(h)),\quad\quad &\text{if } h\in BV(\Omega),\ h\geqslant0\ a.e. \text{ in }\Omega,\\
		&\text{and }\int_{\Omega}h\mathrm{d}x=V,\\
		\\
		+\infty, &\text{otherwise in } L^1(\Omega),
	\end{cases}
\end{equation}
where $ \mathcal{L}^n $ is the $ n $-dimensional Lebesgue measure and $\mathcal{N}(h)=\overline{\{x:h(x)=0\}} $ is the closure of the zero points of $ h $. The definition of $ \int_{\Omega}\sqrt{1+|Dh|^2} $ is shown in \eqref{area_formula_decompositon} which represents the perimeter of the subgraph $ U:=\{(x,t)\in\Omega\times\mathbb{R}:t<h(x)\} $ of $ h $ in $\Omega$.\cite{giusti1984} 
It is necessary to assume $ h\geqslant0 \ a.e. $ in $\Omega$, because the sharp-interface model does not make sense if $ h<0 $.

Now, we present the main theorem about the convergence:
\begin{theorem}{($\Gamma$-convergence)}\label{main_theorem}
	Let $\Omega \subset \mathbb{R}^{n}$ be an open bounded set with Lipschitz boundary, the Young's angle $\theta_e\in(0,\frac{\pi}{2})$, and assume that $g(z)$ satisfies the conditions (1), (2) and ($3'$). Then, the family $ F^{\varepsilon}(h) $ $\Gamma$-converges to $ F(h) $ in $ L^1(\Omega) $, as $\varepsilon\to0$, if $ \mathcal{H}^{n-1}(\partial \mathcal{N}(h))<\infty $, and there is no approximate jump point in $ \{x\in\Omega:h(x)=0\} $.
\end{theorem}

\begin{remark}
	In Theorem \ref{main_theorem}, we assume $\mathcal{H}^{n-1}(\partial \mathcal{N}(h))<\infty$, which has obvious physical meanings if $h$ is smooth: when $n=1$, it implies that the number of contact points is finite; when $n=2$, it implies that the length of contact line is finite.
\end{remark}


\begin{remark}
	The definition of the approximate jump points is given by Definition \ref{def_jump}. If there exists an approximate jump point in $ \{x\in\Omega:h(x)=0\} $, then Young's contact angle at this point must be ${\pi}/{2}$, which is excluded by the assumption of the proposed model (i.e., we consider Young's angle $\theta_e \in (0,{\pi}/{2})$ in this paper).
\end{remark}

By the definition of the $\Gamma$-convergence, it is sufficient to establish the compactness, the lower estimate and the upper estimate. We will show them consequently in the following three subsections.

\subsection{Compactness}

Using Lemma \ref{compact_lemma}, we can establish the compactness under the free energy functional \eqref{func:precursor}.

\begin{theorem}{(Compactness)}\label{compact2}
	Let $\Omega \subset \mathbb{R}^{n}$ be an open bounded set with Lipschitz boundary and the Young's angle $\theta_e\in(0,\frac{\pi}{2})$. Assume $g$ satisfies the conditions (1), (2) and ($3'$). Let $\varepsilon_{j} \rightarrow 0^{+}$and let $\left\{h_{j}\right\} \subset W^{1,1}(\Omega)$ be a sequence such that
	$$
	M:=\sup _{j} F^{\varepsilon_{j}}\left(h_{j}\right)<\infty.
	$$
	Then there exists a subsequence $\left\{h_{j_{k}}\right\}$ of $\left\{h_{j}\right\}$ and $h \in B V(\Omega )$ such that $h_{j_{k}} \rightarrow h$ in $L^{1}(\Omega)$.  Moreover, $\int_\Omega h\mathrm{d}x=V$ and $ h\geqslant0 $ for almost every $ x\in\Omega$.
\end{theorem}
\begin{proof}
	Since $ \gamma^{\varepsilon}(h)\geqslant\gamma_{_{FV}}+S=\gamma_{_{VS}}-\gamma_{_{FS}}>0 $, we have for any $j$ that
	\begin{equation*}
		\begin{aligned}
			F^{\varepsilon_{j}}\left(h_{j}\right)\geqslant(\gamma_{_{VS}}-\gamma_{_{FS}})\int_{\Omega}\sqrt{1+|\nabla h|^2}\mathrm{d}x>\frac{\gamma_{_{VS}}-\gamma_{_{FS}}}{n}\|\nabla h_j\|_{L^1(\Omega)},
		\end{aligned}
	\end{equation*}
	where $\|\nabla h_j\|_{L^1(\Omega)}=\sum_{i=1}^n\int_{\Omega}|\partial_{x_i}h_j|\mathrm{d}x$. Hence
	\[
	\sup_{j}\|\nabla h_j\|_{L^1(\Omega)}<\frac{Mn}{\gamma_{_{VS}}-\gamma_{_{FS}}}<\infty .
	\]
	Because $\Omega$ is bounded, we know that $ \{h_j\} $ is uniformly bounded in $ W^{1,1} $ norm. Therefore, $ \{h_j\} $ is uniformly bounded in BV norm.
	By  Lemma \ref{compact_lemma}, there exists a subsequence $\left\{h_{j_{k}}\right\}_{k=1}^{\infty}$ and a function $h \in B V(\Omega)$ such that
	\[
	h_{j_{k}} \rightarrow h \text { in } L^{1}(\Omega),\quad\quad as \quad k\to+\infty.
	\]
	
	Since $ M:=\sup _{j} F^{\varepsilon_{j}}\left(h_{j}\right)<\infty$, by \eqref{func:precursor} we have $\int_\Omega h_{j_k}\mathrm{d}x=V$ which leads to $\int_\Omega h\mathrm{d}x=V$. The remaining problem is to prove $ h\geqslant0 $ for almost every $ x\in\Omega$. Up to subsequence, let us assume  $ h_{j_{k}} $ converges pointwise to $ h $ for almost every $ x\in\Omega$. Define a set $ A:=\{x\in\Omega:h(x)<0\}$. Since $ \gamma^{\varepsilon_{j_{k}}}(h)>0 $, by Fatou's lemma,
	\begin{equation}\label{ieq:com2}		
		\begin{aligned}
			+\infty&>\liminf_{k\to\infty}F^{\varepsilon_{j_{k}}}(h_{j_{k}})\geqslant\liminf_{k\to\infty}\int_{\Omega}\gamma^{\varepsilon_{j_k}}(h_{j_k})\mathrm{d}x\\
			&\geqslant\int_\Omega\liminf_{k\to\infty}\gamma^{\varepsilon_{j_k}}(h_{j_k})\mathrm{d}x\\
			&\geqslant\int_A \liminf_{k\to\infty}\gamma^{\varepsilon_{j_k}}(h_{j_k})\mathrm{d}x.
		\end{aligned}
	\end{equation}
	For almost every $ x\in A $, we have $\lim\limits_{k\to\infty}\frac{h_{j_k}(x)}{\varepsilon_{j_{k}}}=-\infty$. 	Since $ g $ is continuous and $\lim\limits_{z\to-\infty}g(z)=+\infty$,
	\[
		\liminf_{k\to\infty}\gamma^{\varepsilon_{j_k}}(h_{j_k})=\gamma_{_{FV}}-Sg\bigg(\lim\limits_{k\to\infty}\frac{h_{j_k}}{\varepsilon_{j_{k}}}\bigg)=+\infty,\quad a.e. \ \mbox{in}\ A.
	\]
	Combining \eqref{ieq:com2}, we know $ \mathcal{L}^n(A)=0 $ which implies $ h\geqslant0 \ a.e.$ in $\Omega$.
\end{proof}

\subsection{The lower estimate}

First, we recall a relaxation result (see Theorem 1.1 of Ref. \refcite{fonseca2001}).
\begin{lemma}\label{relaxation_ieq}
	Assume that $f: \Omega \times \mathbb{R} \times \mathbb{R}^{n} \rightarrow[0,+\infty)$ is a Borel integrand, $f(x, h, \cdot)$ is convex in $\mathbb{R}^{n}$, and for all $\left(x_{0}, h_{0}\right) \in \Omega \times \mathbb{R}$ and $\eta>0$ there exists $\delta>0$ such that $f\left(x_{0}, h_{0}, \xi\right)-f(x, h, \xi) \leqslant \eta(1+f(x, h, \xi))$ for all $(x, h) \in \Omega \times \mathbb{R}$ with $\left|x-x_{0}\right|+\left|h-h_{0}\right| \leqslant \delta$ and for all $\xi \in \mathbb{R}^{n} .$ Let $h \in B V_{\mathrm{loc}}(\Omega)$ and let $h_{j} \rightarrow$ h strongly in $L_{\mathrm{loc}}^{1}(\Omega)$, with $h_{j} \in W_{\mathrm{loc}}^{1,1}(\Omega) .$ Then
	$$
	\begin{aligned}
		&\int_{\Omega} f(x, h, \nabla h) \mathrm{d} x+\int_{\Omega} f^{\infty}(x, h, \frac{\mathrm{d}D^{c} h}{\mathrm{d}|D^{c} h|})\mathrm{d}|D^{c} h|+\int_{J_h\cap\Omega} \int_{h^{-}(x)}^{h^{+}(x)} f^{\infty}\left(x, s, \boldsymbol{\nu_{h}}\right) \mathrm{d} s \mathrm{~d} \mathcal{H}^{n-1} \\
		&\quad \leqslant \liminf _{j\to\infty} \int_{\Omega} f\left(x, h_{j}, \nabla h_{j}\right) \mathrm{d} x,
	\end{aligned}
	$$
	where $ Dh $ is the distributional derivative of $ h $,	$ \nabla h $ is the Radon-Nikodym derivative of  $ Dh $ with respect to Lebesgue measure $ \mathcal{L}^n $, $D^{c} h$ is the Cantor part of $Dh$, $ J_h $ is the jump set, $ h^\pm $ are one-side approximate limits, $ \boldsymbol{\nu}_{h} $ is unit normal vector, and $f^{\infty}$ is the recession function of $f$ given by
	$$
	f^{\infty}(x, h, \xi):=\limsup _{t \rightarrow+\infty} \frac{f(x, h, t \xi)}{t}.
	$$
\end{lemma}

Now, we give the lower estimate.
\begin{theorem}{(The lower estimate)}\label{lower estimate}
	Let $ \Omega \subset \mathbb{R}^{n}$ be an open bounded set with Lipschitz boundary and the Young's angle $\theta_e\in(0,\frac{\pi}{2})$. Assume $g$ satisfies the conditions (1), (2) and ($3'$). If $ \mathcal{H}^{n-1}(\partial \mathcal{N}(h))<\infty $, and there is no approximate jump point in $ \{x\in\Omega:h(x)=0\} $, for any sequence $\{h_{j}\}$ with $h_{j} \rightarrow h$ in $L^{1}(\Omega)$ and $\varepsilon_j\to0^+$ as $ j\to\infty $, we have:
	\[
	\liminf_{j\to\infty} F^{\varepsilon_j}(h_j)\geqslant F(h).
	\]
\end{theorem}

\begin{proof}	
	We only need to prove the result under the assumption that
	$$
	\liminf _{j \rightarrow\infty} F^{\varepsilon_{j}}\left(h_{j}\right)<\infty .
	$$
	Let $\left\{\varepsilon_{j_{k}}\right\}$ be a subsequence of $\left\{\varepsilon_{j}\right\}$ such that
	$$
	\liminf _{j \rightarrow\infty} F^{\varepsilon_{j}}\left(h_{j}\right)=\lim _{k \rightarrow\infty} F^{\varepsilon_{j_{k}}}\left(h_{j_{k}}\right)<\infty .
	$$
	Then $F^{\varepsilon_{j_{k}}}\left(h_{j_{k}}\right)<\infty$ for all $k$ sufficiently large, which implies $h_{j_{k}} \in W^{1,1}\left(\Omega \right)$ and $ \int_\Omega h_{j_k}\mathrm{d}x=V $ for all $k$ sufficiently large. By Theorem \ref{compact2}, we obtain $ h \in BV(\Omega)$, $ \int_\Omega h\mathrm{d}x=V $ and $ h\geqslant0$ for almost every $ x\in\Omega $.
	
	Without loss of generality, we will assume
	$\left\{h_{j}\right\} \subset W^{1,1}\left(\Omega \right)$,  $h \in BV(\Omega )$, $\liminf _{j \rightarrow+\infty} F^{\varepsilon_{j}}\left(h_{j}\right)=\lim _{j \rightarrow+\infty} F^{\varepsilon_{j}}\left(h_{j}\right)<\infty$, $\left\{h_{j}\right\}$ converges to $h$ in $L^{1}\left(\Omega \right)$,  $ \int_\Omega h_{j}\mathrm{d}x=\int_\Omega h\mathrm{d}x=V $, $ h\geqslant0\ a.e. $ in $\Omega$ and $\varepsilon_j<1$.
	
	Since $ g(-z)\geqslant g(z) $ for any $ z\geqslant0 $, we have
	\begin{equation}\label{low1}
		\begin{aligned}
			F^{\varepsilon_j}(h_j)
			&\geqslant\int_{\Omega}\gamma_{_{FV}}\sqrt{1+\left|\nabla h_j\right|^{2}}\mathrm{d}x+\int_{\Omega}(-S) g\left(\frac{|h_j|}{\varepsilon_j}\right)\sqrt{1+\left|\nabla h_j\right|^{2}}\mathrm{d}x
		\end{aligned}
	\end{equation}
	For the second term of (\ref{low1}), by Newton-Leibniz formula and Tonelli theorem,
	\begin{equation}\label{low2}
		\begin{aligned}
			&\int_{\Omega}(-S) g\left(\frac{|h_j|}{\varepsilon_j}\right)\sqrt{1+\left|\nabla h_j\right|^{2}}\mathrm{d}x\\
			=&\int_{\Omega}\bigg(\int_{\{t:0<t<\frac{|h_j(x)|}{\varepsilon_j}\}}(-S)g'(t)\sqrt{1+\left|\nabla h_j\right|^{2}}\mathrm{d}t\bigg)\mathrm{d}x+(-S\cdot g(0))\int_{\Omega}\sqrt{1+\left|\nabla h_j\right|^{2}}\mathrm{d}x\\
			=&\int_0^{\infty}\bigg(\int_{\{x:|h_j(x)|>\varepsilon_j t\}}(-S)g'(t)\sqrt{1+\left|\nabla h_j\right|^{2}}\mathrm{d}x\bigg)\mathrm{d}t+S\int_{\Omega}\sqrt{1+\left|\nabla h_j\right|^{2}}\mathrm{d}x.
		\end{aligned}
	\end{equation}
	Define a set $ B_j=\{x:|h_j(x)|>\varepsilon_j \log\frac{1}{\varepsilon_j}\} $.
	For the first term of (\ref{low2}), since $ g'(t)\geqslant0 $ for $ t\geqslant0 $ and $ S<0 $, we can give its estimate:
	\begin{equation}\label{low3}
		\begin{aligned}
			&\int_0^{\infty}\bigg(\int_{\{x:|h_j(x)|>\varepsilon_j t\}}(-S)g'(t)\sqrt{1+\left|\nabla h_j\right|^{2}}\mathrm{d}x\bigg)\mathrm{d}t\\
			\geqslant&\int_{0}^{\log\frac{1}{\varepsilon_j}}\int_{\{x:|h_j(x)|>\varepsilon_j t\}}(-S)g'(t)\sqrt{1+\left|\nabla h_j\right|^{2}}\mathrm{d}x\mathrm{d}t\\
			\geqslant&\int_0^{\log\frac{1}{\varepsilon_j}}(-S)g'(t)\mathrm{d}t\cdot\int_{B_j}\sqrt{1+\left|\nabla h_j\right|^{2}}\mathrm{d}x\\
			=&-Sg(\log\frac{1}{\varepsilon_j})\int_{B_j}\sqrt{1+\left|\nabla h_j\right|^{2}}\mathrm{d}x-S\int_{B_j}\sqrt{1+\left|\nabla h_j\right|^{2}}\mathrm{d}x
		\end{aligned}
	\end{equation}
	Combining (\ref{low1}), (\ref{low2}) and (\ref{low3}):
	\begin{equation*}
		\begin{aligned}
			F^{\varepsilon_j}(h_j)\geqslant&
			\int_{\Omega}(\gamma_{_{FV}}+S)\sqrt{1+\left|\nabla h_j\right|^{2}}\mathrm{d}x-Sg(\log\frac{1}{\varepsilon_j})\int_{B_j}\sqrt{1+\left|\nabla h_j\right|^{2}}\mathrm{d}x\\&-S\int_{B_j}\sqrt{1+\left|\nabla h_j\right|^{2}}\mathrm{d}x\\
			:=&I+II+III.
		\end{aligned}
	\end{equation*}
	Define an operator:
	\[
	T(h):=\int_{\Omega}\sqrt{1+|\nabla h|^2}\mathrm{d}x+|D^{c} h|(\Omega)+\int_{J_h\cap\Omega}|h^+-h^-|  \mathrm{d} \mathcal{H}^{n-1}.
	\]
	From (\ref{area_formula_decompositon}), we can rewrite $ T(h) $ as:
	\[
	T(h)=\int_\Omega\sqrt{1+|D h|^2}.
	\]
	Let $ f(x,h,\xi)=\sqrt{1+\xi^2} $ which satisfies all the assumptions in Lemma \ref{relaxation_ieq}, then the recession function is $ f^\infty(x,h,\xi)=|\xi| $. By polar decomposition (c.f. Ref. \refcite{ambrosio2000}),
	\[
	    \bigg|\frac{\mathrm{d}D^ch}{\mathrm{d}|D^ch|}\bigg|=1.
	\]
	As a consequence of Lemma \ref{relaxation_ieq}, we have
	\[
	\liminf_{j\to\infty}\int_{\Omega}\sqrt{1+|\nabla h_j|^2}\mathrm{d}x\geqslant T(h).
	\]
	Because $ \gamma_{_{FV}}+S=\gamma_{_{VS}}-\gamma_{_{FS}}>0 $,
	\begin{equation}\label{estimateI}
		\liminf_{j\to\infty}I\geqslant(\gamma_{_{FV}}+S)T(h).
	\end{equation}
	
	Due to the assumption that $\lim _{j \rightarrow\infty} F^{\varepsilon_{j}}\left(h_{j}\right)<+\infty$, there exists a positive constant $ M_0 $ such that $ F^{\varepsilon_{j}}\left(h_{j}\right)\leqslant M_0 $ for sufficiently large $j$.
	Hence,
	\[
	M_0\geqslant F^{\varepsilon_{j}}\left(h_{j}\right)\geqslant(\gamma_{_{VS}}-\gamma_{_{FS}})\int_{\Omega}\sqrt{1+|\nabla h_j|^2}\mathrm{d}x.
	\]
	For $ II $, because $ \gamma_{_{VS}}-\gamma_{_{FS}}>0 $,
	\begin{equation}\label{estimateII}
		|II|\leqslant Sg(\log\frac{1}{\varepsilon_j})\int_{\Omega}\sqrt{1+\left|\nabla h_j\right|^{2}}\mathrm{d}x\leqslant Sg(\log\frac{1}{\varepsilon_j})\frac{M_0}{\gamma_{_{VS}}-\gamma_{_{FS}}}.
	\end{equation}
	Let $\varepsilon_j\to0^+$, we deduce that $ II\rightarrow 0. $
	
	The remaining problem is to estimate III. Let $ f(x,h,\xi)=\chi_{_{h>0}}(x)\sqrt{1+\xi^2} $ in Lemma \ref{relaxation_ieq}, where
	\[
		\chi_{_{h>0}}(x)=\begin{cases}
		1, \quad\quad &\text{if } \, h>0,\\
		0, &\text{if } \, h\leqslant0.
	\end{cases}
	\]
	It can be shown that this function also satisfies all the assumptions in Lemma \ref{relaxation_ieq}. In fact, $ f(x,h,\cdot) $ is a convex function since $ \frac{\partial^2f}{\partial\xi^2}=\chi_{_{h>0}}(x)\frac{1}{(1+\xi^2)^\frac{3}{2}}\geqslant0 $. If $ h_0\leqslant0 $, for all $(x, h) \in \Omega \times \mathbb{R}$, we have
	\[
		f\left(x_{0}, h_{0}, \xi\right)-f(x, h, \xi)=0-f(x, h, \xi)\leqslant0 \leqslant \eta(1+f(x, h, \xi)).
	\]
	If $ h_0>0 $, there exists $\delta>0$ such that $ h>0 $ for all $(x, h) \in \Omega \times \mathbb{R}$ with $\left|x-x_{0}\right|+\left|h-h_{0}\right| \leqslant \delta$. So,
	\[
	f\left(x_{0}, h_{0}, \xi\right)-f(x, h, \xi)=\sqrt{1+\xi^2}-\sqrt{1+\xi^2}=0 \leqslant \eta(1+f(x, h, \xi)).
	\]	
	Now that $f(x,h,\xi)$ satisfies the assumptions in Lemma \ref{relaxation_ieq} with its recession function given by $ f^\infty(x,h,\xi)=\chi_{_{h>0}}(x)|\xi| $. Let $ \bar{h}_j=h_j-\varepsilon_j\log\frac{1}{\varepsilon_j} $, then $ \bar{h}_j\to h $ in $ L^1(\Omega) $. By Lemma \ref{relaxation_ieq}, we have
	\begin{equation}\label{barh_estimate1}
		\begin{aligned}
			&\liminf_{j\to\infty} \int_{\{x:\bar{h}_j(x)>0\}}\sqrt{1+\left|\nabla \bar{h}_j\right|^{2}}\mathrm{d}x\\
			\geqslant&\int_{\{x:h(x)>0\}}\sqrt{1+|\nabla h|^2}\mathrm{d}x+|D^{c} h|(\{x:h(x)>0\})+\int_{J_h\cap\{x:h(x)>0\}}|h^+-h^-|  \mathrm{d} \mathcal{H}^{n-1}.
		\end{aligned}
	\end{equation}
	Similarly, let $ f(x,h,\xi)=\chi_{_{h<0}}(x)\sqrt{1+\xi^2} $ and  $ \tilde{h}_j=h_j+\varepsilon_j\log\frac{1}{\varepsilon_j} $, then
	\begin{equation}\label{barh_estimate2}
		\begin{aligned}
			&\liminf_{j\to\infty}\int_{\{x:\tilde{h}_j(x)<0\}}\sqrt{1+\left|\nabla \tilde{h}_j\right|^{2}}\mathrm{d}x\\
			\geqslant&\int_{\{x:h(x)<0\}}\sqrt{1+|\nabla h|^2}\mathrm{d}x+|D^{c} h|(\{x:h(x)<0\})+\int_{J_h\cap\{x:h(x)<0\}}|h^+-h^-|  \mathrm{d} \mathcal{H}^{n-1}.
		\end{aligned}
	\end{equation}
	Combining \eqref{barh_estimate1} and \eqref{barh_estimate2}, we obtain
	\begin{equation}\label{estimateIII}
		\begin{aligned}
		&\liminf_{j\to\infty}III\\
		\geqslant&\liminf_{j\to\infty}(-S)\int_{\{x:\bar{h}_j(x)>0\}}\sqrt{1+\left|\nabla \bar{h}_j\right|^{2}}\mathrm{d}x+\liminf_{j\to\infty}(-S)\int_{\{x:\tilde{h}_j(x)<0\}}\sqrt{1+\left|\nabla \tilde{h}_j\right|^{2}}\mathrm{d}x\\
		\geqslant&-S\bigg(\int_{\{x:|h(x)|>0\}}\sqrt{1+|\nabla h|^2}\mathrm{d}x+|D^{c} h|(\{x:|h(x)|>0\})\\
		&\quad\quad\quad+\int_{J_h\cap\{x:|h(x)|>0\}}|h^+-h^-|  \mathrm{d} \mathcal{H}^{n-1}\bigg).
		\end{aligned}
	\end{equation}

	For the absolutely continuous part, we have $ \nabla h=0 $ a.e. in $ \{x:h(x)=0\} $, so
	\begin{equation}\label{ieq_ac}
		\int_{\{x:h(x)=0\}}\sqrt{1+|\nabla h|^2}\mathrm{d}x=\mathcal{L}^{n}(\{x:h(x)=0\}).
	\end{equation}

	For the Cantor part, because $ \lim _{R \downarrow 0} R^{-n}|D h|\left(B_{R}(x)\right)=0 $ for any $ x $ in the interior of $ \{x:h(x)=0\} $. From  Lemma \ref{prop_singular},
	\[
	|D^sh|\Big(\{x:h(x)=0\}^\circ\Big)=0,
	\]
	which implies $ |D^ch|\Big(\{x:h(x)=0\}^\circ\Big)=0$.
	
	Moreover, $ |D^ch|\Big(\partial\mathcal{N}(h)\Big)=0 $ as a result of $ \mathcal{H}^{n-1}(\partial \mathcal{N}(h))<\infty $ and Lemma \ref{cantor_lemma}.
	Therefore,
	\begin{equation}\label{ieq_cantor}
		\begin{aligned}
			&		|D^ch|\Big(\{x:h(x)=0\}\Big)=|D^ch|\Big(\partial\{x:h(x)=0\}\Big)\leqslant|D^ch|\Big(\partial\mathcal{N}(h)\Big)=0.
		\end{aligned}
	\end{equation}

	The assumption that there is no approximate jump point in $ \{x:h(x)=0\} $ implies $ J_h\cap\{x\in\Omega:h=0\}=\emptyset $, hence
	\begin{equation}\label{eq_jump}
		\int_{J_h\cap\{x:h(x)=0\}}|h^+-h^-|  \mathrm{~d} \mathcal{H}^{n-1}=0.
	\end{equation}
	
	Combining (\ref{estimateI}), (\ref{estimateII}) (\ref{estimateIII}), (\ref{ieq_ac}), (\ref{ieq_cantor}) and \eqref{eq_jump}, and noticing $S<0$, we arrive at
	\[
	\begin{aligned}
		&\liminf_{j\to\infty} F^{\varepsilon_j}(h_j)\\
		\geqslant&\gamma_{_{FV}}T(h)+S\int_{\{x:h(x)=0\}}\sqrt{1+|\nabla h|^2}\mathrm{d}x+S\cdot|D^ch|\Big(\{x:h(x)=0\}\Big)\\&+S\int_{J_h\cap\{x:h=0\}}|h^+-h^-|  \mathrm{~d} \mathcal{H}^{n-1}\\
		=&\gamma_{_{FV}}T(h)+S\mathcal{L}^{n}(\{x:h(x)=0\})+0+0\\
		\geqslant &F(h).
	\end{aligned}
	\]
\end{proof}

\subsection{The upper estimate}

Let $ \operatorname{Supp}(h):=\overline{\{x:h(x)\neq0\}}$. We first give Lemma \ref{lemma_upper} whose proof is similar to Ref. \refcite{bildhauer03} and is left to \ref{proof4.2}.

\begin{lemma}\label{lemma_upper}
	Let $h \in B V\left(\Omega\right)$ and $ h\geqslant0$ for almost every $ x\in\Omega $. There is a non-negative sequence $\left\{h_{j}\right\}$ in $C^{\infty}\left(\Omega\right)$ satisfying
	$$
	\begin{aligned}
		&\lim _{j \rightarrow \infty} \int_{\Omega}\left|h_{j}-h\right| \mathrm{d}x=0, \\
		&\lim _{j \rightarrow \infty} \int_{\Omega} \sqrt{1+\left|\nabla h_{j}\right|^{2}} \mathrm{d}x=\int_{\Omega} \sqrt{1+|D h|^{2}},\\
		&\lim _{j \rightarrow \infty}\mathcal{L}^n(\operatorname{Supp}(h_j))\leqslant\mathcal{L}^n(\operatorname{Supp}(h)).
	\end{aligned}
	$$
\end{lemma}

\begin{theorem}{(The upper estimate)}\label{upper estimate}
	Let $ \Omega \subset \mathbb{R}^{n}$ be an open bounded set with Lipschitz boundary and the Young's angle $\theta_e\in(0,\frac{\pi}{2})$. Assume $g$ satisfies the conditions (1) and (2). For any $h \in L^1(\Omega )$ satisfying $ \mathcal{H}^{n-1}(\partial \mathcal{N}(h))<\infty $, there exists a sequence $ \{h_j\} $, such that $ h_j\to h $ in $ L^1(\Omega) $, $\int_\Omega h_j\mathrm{d}x=V$ and
	\[
	\limsup_{j\to\infty} F^{\varepsilon_j}(h_j)\leqslant F(h).
	\]
\end{theorem}
\begin{proof}
    We only need to consider the case $F(h)<+\infty$ which means we have $h \in B V(\Omega )$ satisfying $ \int_{\Omega}h\mathrm{d}x=V $ and $ h\geqslant0 \ a.e.$ in $\Omega$. By Lemma \ref{lemma_upper}, for any $ \varepsilon_j\to0^+ $, there exists a non-negative sequence $ \{\tilde{h}_j\} $ in $C^\infty(\Omega)$, such that
	\begin{equation}\label{upper_lengthestimate}
		\begin{aligned}
			&\int_\Omega|\tilde{h}_j-h|\mathrm{d}x\leqslant\varepsilon_j^2,\\
			&\lim _{j \rightarrow \infty} \int_{\Omega} \sqrt{1+\left|\nabla \tilde{h}_{j}\right|^{2}} \mathrm{d}x=\int_{\Omega} \sqrt{1+|D h|^{2}},\\
			&\lim _{j \rightarrow \infty}\mathcal{L}^n(\operatorname{Supp}(\tilde{h}_j))\leqslant\mathcal{L}^n(\operatorname{Supp}(h)).
		\end{aligned}
	\end{equation}

	Because $ S<0 $ and $ g\left(\frac{\tilde{h}_j}{\varepsilon_j}\right)\leqslant0 $, we have
	\[
	\int_{\Omega}(-S) g\left(\frac{\tilde{h}_j}{\varepsilon_j}\right)\sqrt{1+\left|\nabla \tilde{h}_j\right|^{2}}\mathrm{d}x\leqslant\int_{\Omega}(-S) g\left(\frac{\tilde{h}_j}{\varepsilon_j}\right)\mathrm{d}x.
	\]
	Since $ g'(t)\geqslant0 $ for $ t\geqslant0 $, by Newton-Leibniz formula and Tonelli theorem, we have
	\begin{equation}\label{upper_estimate1}
		\begin{aligned}
			&\int_{\Omega}(-S) g\left(\frac{\tilde{h}_j}{\varepsilon_j}\right)\mathrm{d}x\\
			=&\int_{\Omega}\bigg(\int_{\{t:0<t<\frac{\tilde{h}_j(x)}{\varepsilon_j}\}}(-S)g'(t)\mathrm{d}t\bigg)\mathrm{d}x+(-Sg(0))\mathcal{L}^n(\Omega)\\
			=&\int_{0}^{\infty}\bigg((-S)g'(t)\cdot\mathcal{L}^n(\{x\in\Omega:\frac{\tilde{h}_j(x)}{\varepsilon_j}>t\})\bigg)\mathrm{d}t+S\cdot \mathcal{L}^n(\Omega)\\
			\leqslant&\mathcal{L}^n\bigg(\{x\in\Omega:\frac{\tilde{h}_j(x)}{\varepsilon_j}>0\}\bigg)\int_{0}^{\infty}(-S)g'(t)\mathrm{d}t+S\cdot \mathcal{L}^n(\Omega)\\
			=&-S\cdot \mathcal{L}^n\bigg(\{x\in\Omega:\tilde{h}_j(x)>0\}\bigg)+S\cdot \mathcal{L}^n(\Omega)\\
			\leqslant&-S\mathcal{L}^n(\operatorname{Supp}(\tilde{h}_j))+S\cdot \mathcal{L}^n(\Omega).
		\end{aligned}
	\end{equation}
	
	Define the extended functional of $F^\varepsilon$ on $W^{1,1}(\Omega)$ as
    \[
		G^{\varepsilon}(h)=\int_{\Omega}\bigg(\gamma_{_{FV}}-S g\left(\frac{h}{\varepsilon}\right)\bigg)\sqrt{1+\left|\nabla h\right|^{2}}\mathrm{d}x.
	\]
	Then $G^{\varepsilon}$ coincides with $F^{\varepsilon}$ on the set $\{h\in W^{1,1}(\Omega): \int_\Omega h\mathrm{d}x=V\}$.

	Combining (\ref{upper_lengthestimate}) and (\ref{upper_estimate1}), we obtain
	\[	
	\begin{aligned}
		\limsup_{j \rightarrow \infty}G^{\varepsilon_j}(\tilde{h}_j)&=\limsup_{j \rightarrow \infty}\int_{\Omega}\bigg(\gamma_{_{FV}}-S g\big(\frac{\tilde{h}_j}{\varepsilon_j}\big)\bigg)\sqrt{1+\left|\nabla \tilde{h}_j\right|^{2}}\mathrm{d}x\\
		&\leqslant\gamma_{_{FV}}\lim_{j \rightarrow \infty}\int_{\Omega}\sqrt{1+\left|\nabla \tilde{h}_j\right|^{2}}\mathrm{d}x-S\lim_{j \rightarrow \infty} \mathcal{L}^n(\operatorname{Supp}(\tilde{h}_j))+S\cdot \mathcal{L}^n(\Omega)\\
		&\leqslant\gamma_{_{FV}}\int_{\Omega} \sqrt{1+|D h|^{2}}-S\cdot\mathcal{L}^n(\operatorname{Supp}(h))+S\cdot \mathcal{L}^n(\Omega)\\
		&=\gamma_{_{FV}}\int_{\Omega} \sqrt{1+|D h|^{2}}+S\cdot\mathcal{L}^n(\Omega\backslash\operatorname{Supp}(h))\\
		&= \gamma_{_{FV}}\int_{\Omega} \sqrt{1+|D h|^{2}}+S\cdot\mathcal{L}^n(\mathcal{N}(h))\\
		&=F(h).
	\end{aligned}
	\]
    where in the second to last equality we have used the fact $ \mathcal{L}^n(\partial \mathcal{N}(h))=0 $, which is a result of $ \mathcal{H}^{n-1}(\partial \mathcal{N}(h))<\infty $.
	
	To obtain a recovery sequence for the original functional $F^\varepsilon$ in the upper estimate, we only need to modify $ \tilde{h}_j $ to obtain $ h_j\in W^{1,1}(\Omega) $ such that $ h_j\to h $ in $ L^1(\Omega) $, $ \int_\Omega h_j\mathrm{d}x=V $ and $ \lim\limits_{j \rightarrow \infty}G^{\varepsilon_j}(\tilde{h}_j)=\lim\limits_{j \rightarrow \infty}F^{\varepsilon_j}(h_j) $.
	
	Let $ d_j=\int_\Omega h\mathrm{d}x-\int_\Omega \tilde{h}_j\mathrm{d}x$. Then \eqref{upper_lengthestimate} implies
	\begin{equation}\label{d_estimate}
		|d_j|\leqslant\int_\Omega |h-\tilde{h}_j|\mathrm{d}x\leqslant\varepsilon_j^2.
	\end{equation}
	Define $ h_j:=\tilde{h}_j+\frac{d_j}{\mathcal{L}^n(\Omega)} \in W^{1,1}(\Omega)$. We have $ h_j\to h $ in $ L^1(\Omega) $,  $ \int_\Omega h_j=\int_\Omega h\mathrm{d}x=V $, $ \nabla h_j=\nabla \tilde{h}_j $, and
	\begin{equation}\label{upper_last}
	\begin{aligned}
		&\limsup\limits_{j\to\infty}|G^{\varepsilon_j}(\tilde{h}_j)-F^{\varepsilon_j}(h_j)|\\
		=&\limsup\limits_{j\to\infty}(-S)\bigg|\int_\Omega\bigg(g\big(\frac{\tilde{h}_j}{\varepsilon_j}\big)-g\big(\frac{h_j}{\varepsilon_j}\big)\bigg)\sqrt{1+|\nabla\tilde{h}_j|^2}\mathrm{d}x\bigg|\\
		\leqslant&\limsup\limits_{j\to\infty}(-S)\int_\Omega |g'(\zeta)|\frac{|d_j|}{\varepsilon_j\mathcal{L}^n(\Omega)}\sqrt{1+|\nabla\tilde{h}_j|^2}\mathrm{d}x,
	\end{aligned}
	\end{equation}
	where $ \zeta $ is between $ \frac{\tilde{h}_j}{\varepsilon_j} $ and $ \frac{h_j}{\varepsilon_j} $. From \eqref{d_estimate}, we know
	\begin{equation}\label{d2_estimate}
		\frac{|d_j|}{\varepsilon_j\mathcal{L}^n(\Omega)}\leqslant \frac{\varepsilon_j}{\mathcal{L}^n(\Omega)}.
	\end{equation}
	Since $ \tilde{h}_j\geqslant0 $ and \eqref{d2_estimate}, we know $ \frac{\tilde{h}_j}{\varepsilon_j}\in [0,+\infty) $ and
	\[
	\frac{h_j}{\varepsilon_j}=\frac{\tilde{h}_j}{\varepsilon_j}+\frac{d_j}{\varepsilon_j\mathcal{L}^n(\Omega)}\in[-\frac{\varepsilon_j}{\mathcal{L}^n(\Omega)},+\infty).
	\]
	Then we have $ \zeta\in(-\frac{\max_j\varepsilon_j}{\mathcal{L}^n(\Omega)},+\infty) $. Since $ g'(z) $ is continuous and $ \lim\limits_{z\to+\infty}g'(z)=0 $, there exists a constant $ M_1 $ such that $ |g'(\zeta)|\leqslant M_1 $.
	
	Since $ \lim\limits_{j\to+\infty}\int_\Omega\sqrt{1+|\nabla\tilde{h}_j|^2}\mathrm{d}x=\int_\Omega\sqrt{1+|Dh|^2}\mathrm{d}x<\infty $, 
    \eqref{upper_last} becomes
	\[
	\begin{aligned}
	\limsup\limits_{j\to\infty}|G^{\varepsilon_j}(\tilde{h}_j)-F^{\varepsilon_j}(h_j)|
	\leqslant \lim\limits_{j\to\infty}(-S)M_1\frac{\varepsilon_j}{\mathcal{L}^n(\Omega)}\int_\Omega\sqrt{1+|\nabla\tilde{h}_j|^2}\mathrm{d}x=0,
	\end{aligned}
	\]
	which leads to $ \limsup\limits_{j\to\infty}G^{\varepsilon_j}(\tilde{h}_j)=\limsup\limits_{j\to\infty}F^{\varepsilon_j}(h_j) $. Therefore, $ \{h_j\} $ is a recovery sequence and the upper estimate follows.
\end{proof}

In summary, combining Theorem \ref{lower estimate} and Theorem \ref{upper estimate}, we complete the proof of the Theorem \ref{main_theorem}.

\section{Numerical results}

In this section, we present some numerical experiments to verify the theoretical results (\ref{eq:const-curv}), (\ref{eq:height}) and (\ref{eq:inner}) from the asymptotic analysis in Section 3. For simplicity, we only perform numerical simulations when $\Omega$ is a one-dimensional region.

Let $\Omega=[-1,1]$ be a computational domain. We divide the domain into $N$ equal subintervals $[x_j,x_{j+1}]$,
i.e., $ x_j = -1 + j\Delta x $ and $ h_j=h(x_j) $, $j=0,1,2,\ldots,N$, with $\Delta x=\frac{2}{N}$. The Euler-Lagrange equation \eqref{eq:equilib} is discretized using the central difference scheme, and the volume constraint $ \int_\Omega h(x)\mathrm{d}x=V $ is discretized by using Simpson's formula. Then, we obtain a nonlinear system
$ G(h_1,h_2,\cdots,h_{N+1},\lambda)=0 $. We use the modified Newton's method to solve it numerically. More precisely, we employ the traditional Newton's method combining backtracking line search with Armijo rule.\cite{andrei2006acceleration} The Armijo parameter is chosen as $ 10^{-4} $ and the backtrack coefficient is fixed at 0.5. To avoid the step size being too small, we set its lower bound to be $10^{-4}$. The stopping criterion of the algorithm is $ ||G||_\infty<10^{-8}$ (if the stopping criterion is changed to $||h^{(k+1)}-h^{(k)}||_{\infty}<10^{-8}$, we obtain almost the same results). We note that each Newton's step can be solved very efficiently due to the sparsity of the Jacobian matrix of $G$ (which is close to a tridiagonal matrix).

The initial shape is given by
\[
h(x)=
\begin{cases}
	\sqrt{0.41-x^2}-0.4,\quad & \text{if}~ -0.5\leqslant x\leqslant0.5;\\
	0,&\text{otherwise}.
\end{cases}
\]
Choose $\gamma_{_{FV}}=1$, then $S=\cos\theta_e-1$, where $\theta_e\in (0,\pi/2)$ is Young's contact angle. We take $g(z)=e^{-z}-2e^{-\frac{z}2}$,
and consider two different Young's angle $\theta_e={\pi}/{3}$ and ${\pi}/{6}$ in the simulations. Furthermore, we choose a very fine mesh $N=2048$ in all the numerical simulations in order to make numerical errors as small as possible.

Fig.~\ref{eps} shows the relation between the (numerical) equilibrium shapes produced by the regularized model and the equilibrium shape (shown in red solid line)
of the original model determined by the Winterbottom construction by decreasing the regularization parameters $\varepsilon$ under two different Young's angels $\theta_e={\pi}/{3}, {\pi}/{6}$, respectively. As shown in the figure, we can clearly observe the convergence results of the regularized model when we gradually decrease $\varepsilon$ from $0.01$ to $0.0025$. Furthermore, Fig.~\ref{curvature} also depicts the convergence results of the curvature along the equilibrium curve produced from the regularized model to the original model, when we decrease $\varepsilon$. From the figure, we observe that the curvature function along the curve gradually becomes a constant,
especially for the position which is away from the contact point. This is consistent with the constant curvature of equilibrium shapes by the minimal surface theory, and also agrees with the asymptotic analysis result \eqref{eq:const-curv}.

\begin{figure}[htpb]
	\centering
	\subfloat{\includegraphics[width=10cm]{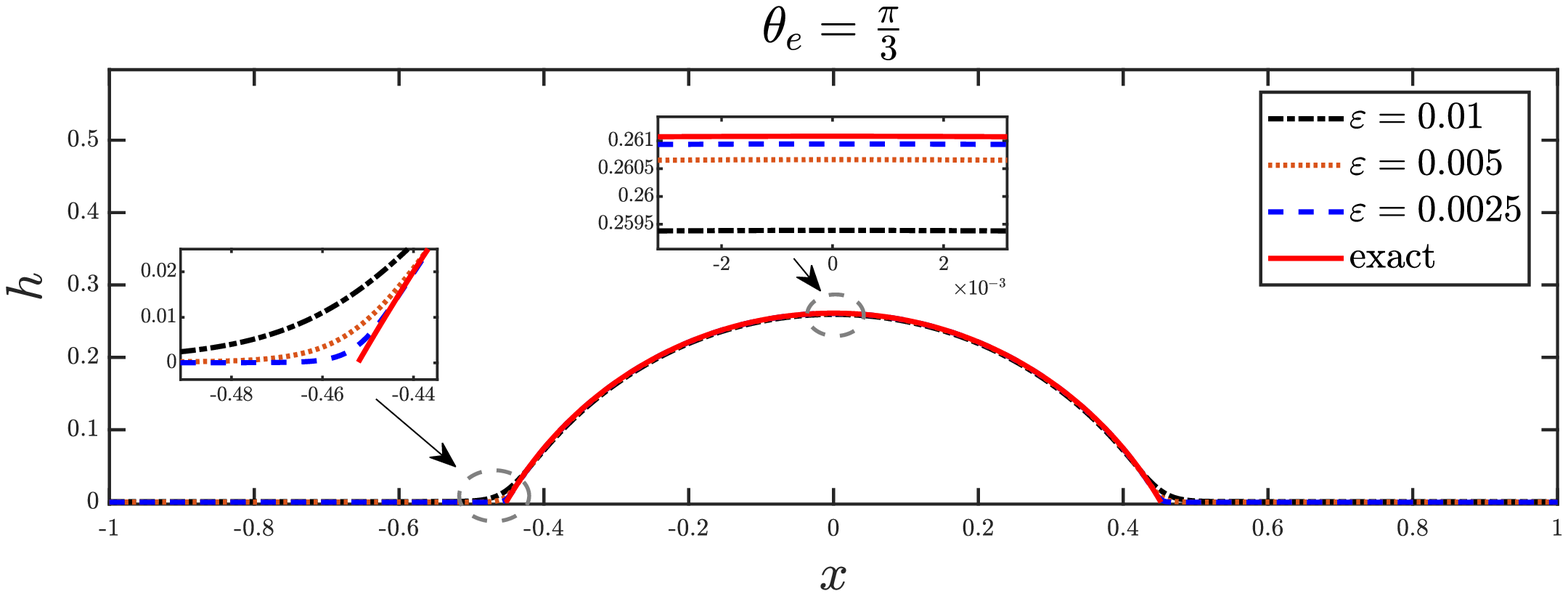}}\\
	\subfloat{\includegraphics[width=10cm]{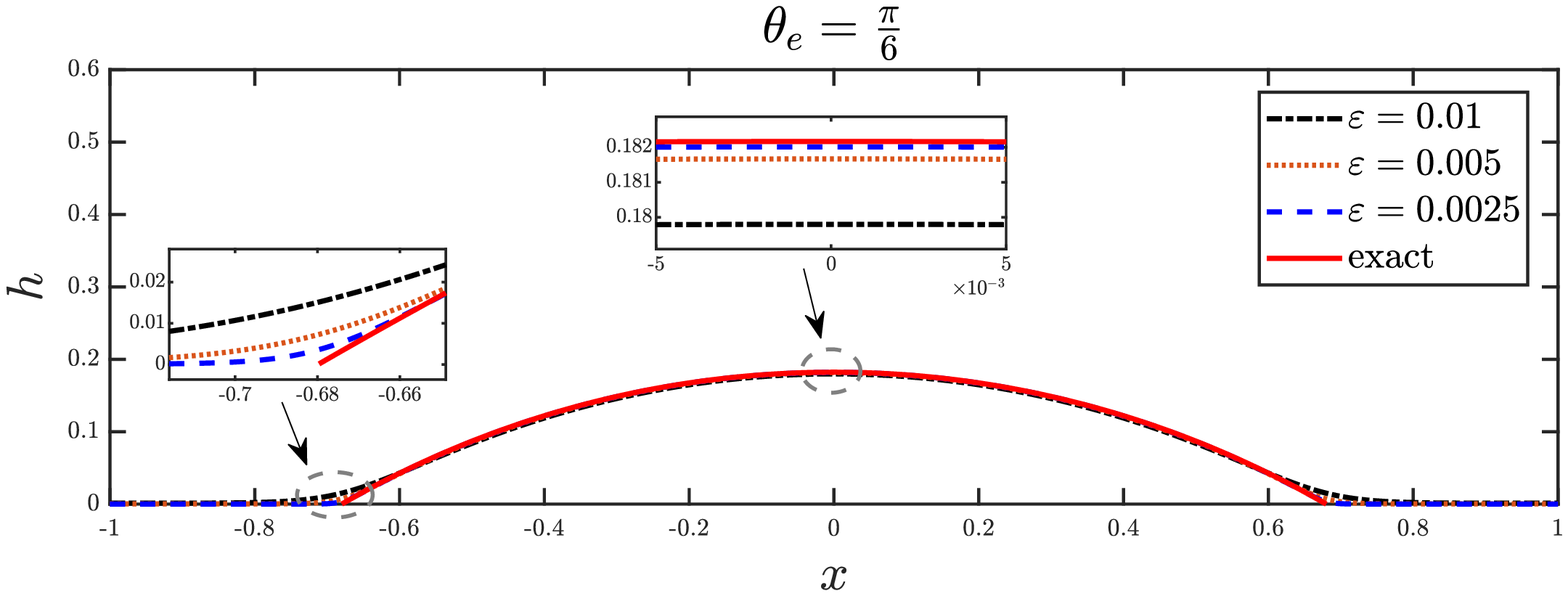}}	
	\caption{\label{eps}Convergence results of equilibrium shapes produced from the regularized model to the original model (in red solid line) by decreasing the regularization parameters $\varepsilon$ under two different Young's angels: $\theta_e={\pi}/{3}$ (top) and $\theta_e={\pi}/{6}$ (bottom), where the zoom-in plots are also shown for a better observation.}
\end{figure}

\begin{figure}[htb]
	\centering
	\subfloat{\includegraphics[width=6cm]{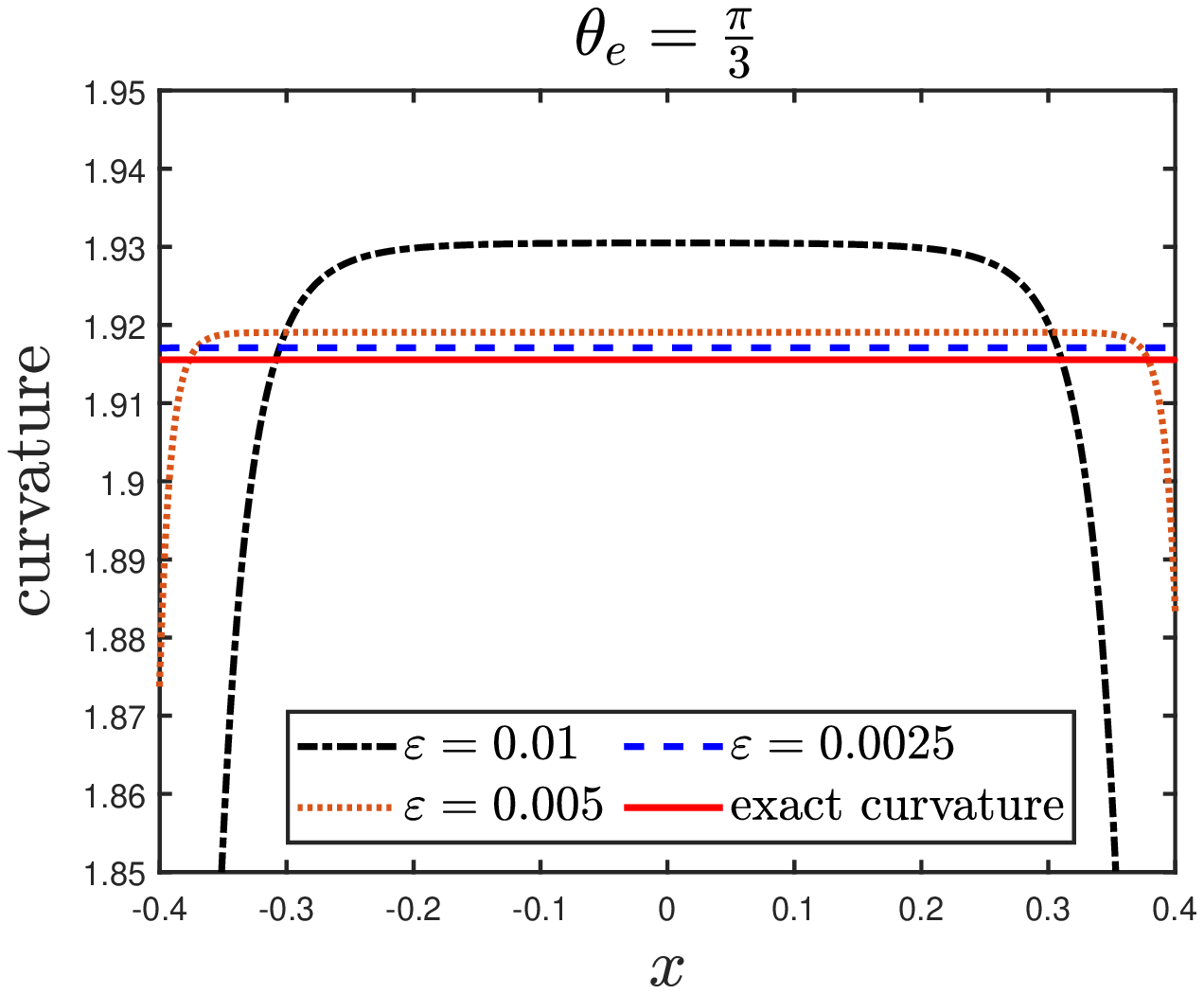}}
	\subfloat{\includegraphics[width=6cm]{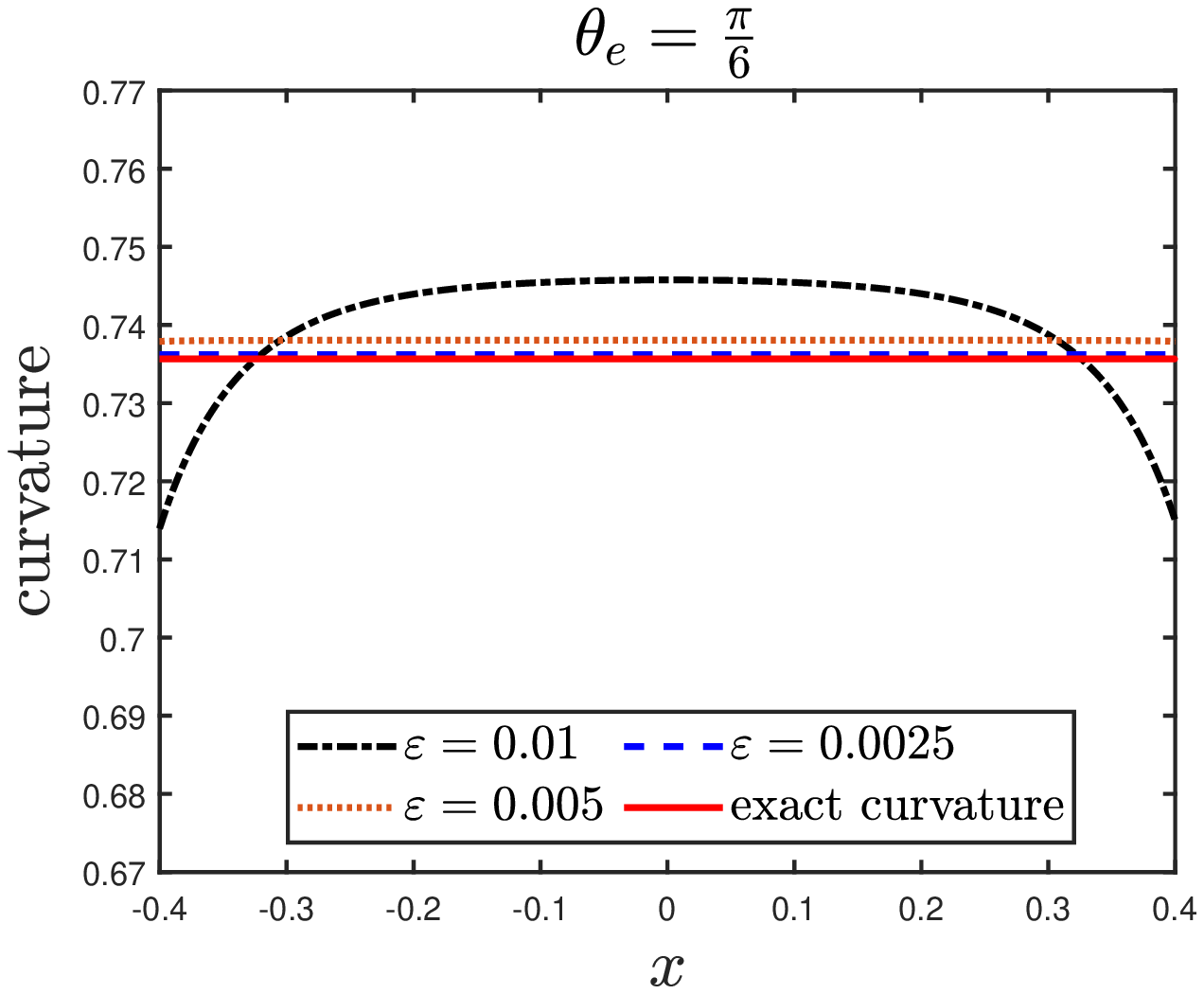}}
	\caption{\label{curvature} Convergence results of the curvature along the equilibrium curve produced from the regularized model to the original model (in red solid line) by decreasing $\varepsilon$ under two different Young's angels: $\theta_e={\pi}/{3}$ (left) and $\theta_e={\pi}/{6}$ (right).}
\end{figure}

In the regularized model, the appearance of the precursor layer has brought lots of advantages for theoretical analysis and numerical simulations, because it transforms the original moving boundary problem into a fixed domain problem. Therefore, the relation between the height of the precursor layer and the small regularization parameter $\varepsilon$ is very essential. Both the asymptotic analysis and numerical results indicate that there exists a constant height of the precursor layer which is away from the contact
point. Thus, the height of the precursor can be numerically calculated at an appropriate point (e.g., we choose the value of $h(x)$ at the point $x=-1$ in the numerical simulations). Fig.~\ref{heightlog}(a) depicts the log-log plot of the precursor height as a function of $\varepsilon$ when Young's angle $\theta_e={\pi}/{3}$.
As clearly shown by the figure, we can see that the precursor height decreases to zero at the second-order rate as $\varepsilon$ goes to zero. Furthermore,
Fig.~\ref{heightlog}(b) clearly shows the coefficient of the second-order rate can be perfectly predicted by the formula \eqref{eq:height}, which is given
by our asymptotic analysis in Section 3.

\begin{figure}[htb]
	\centering
    \subfloat{\includegraphics[width=6cm]{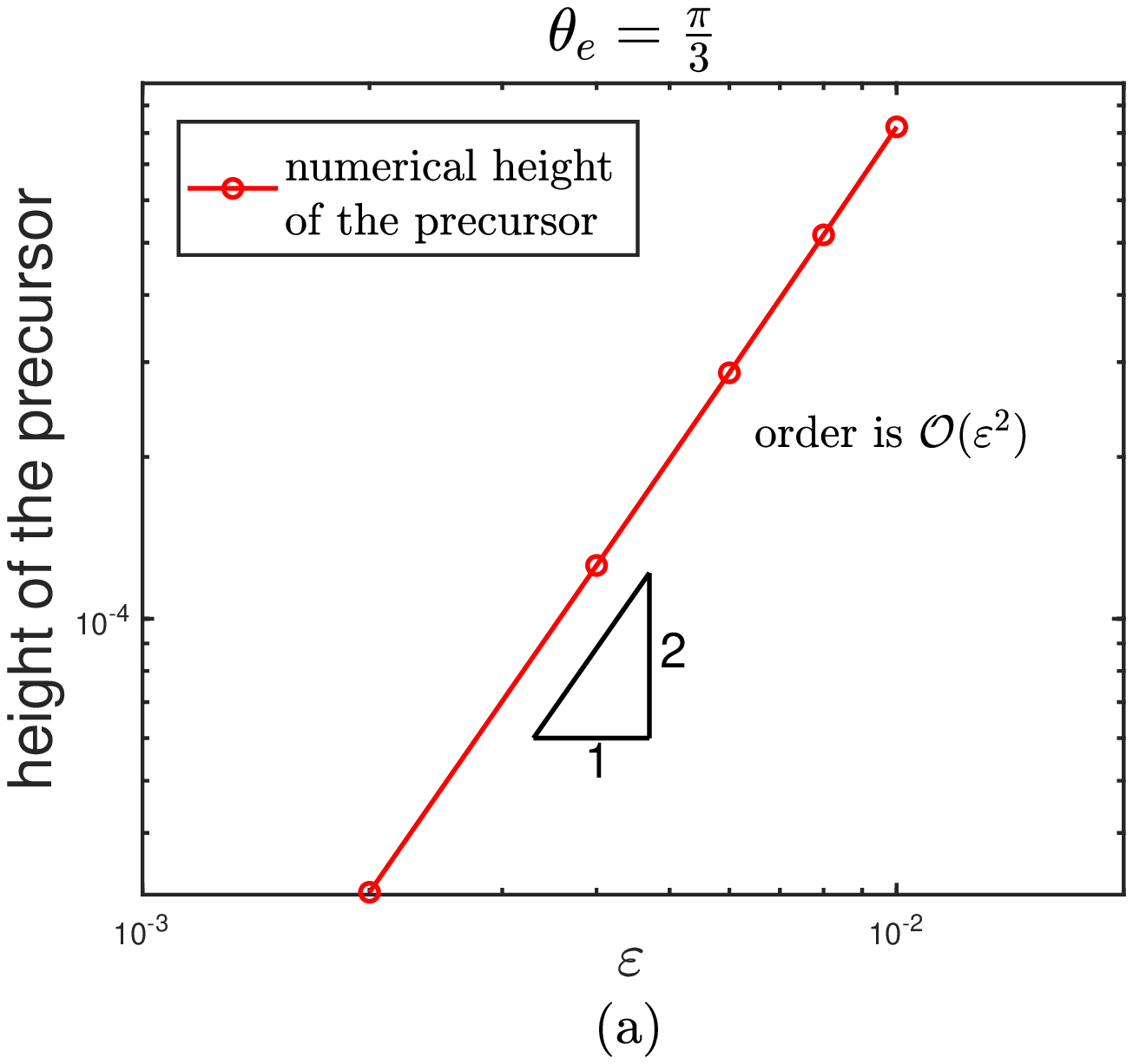}}
	\subfloat{\includegraphics[width=6cm]{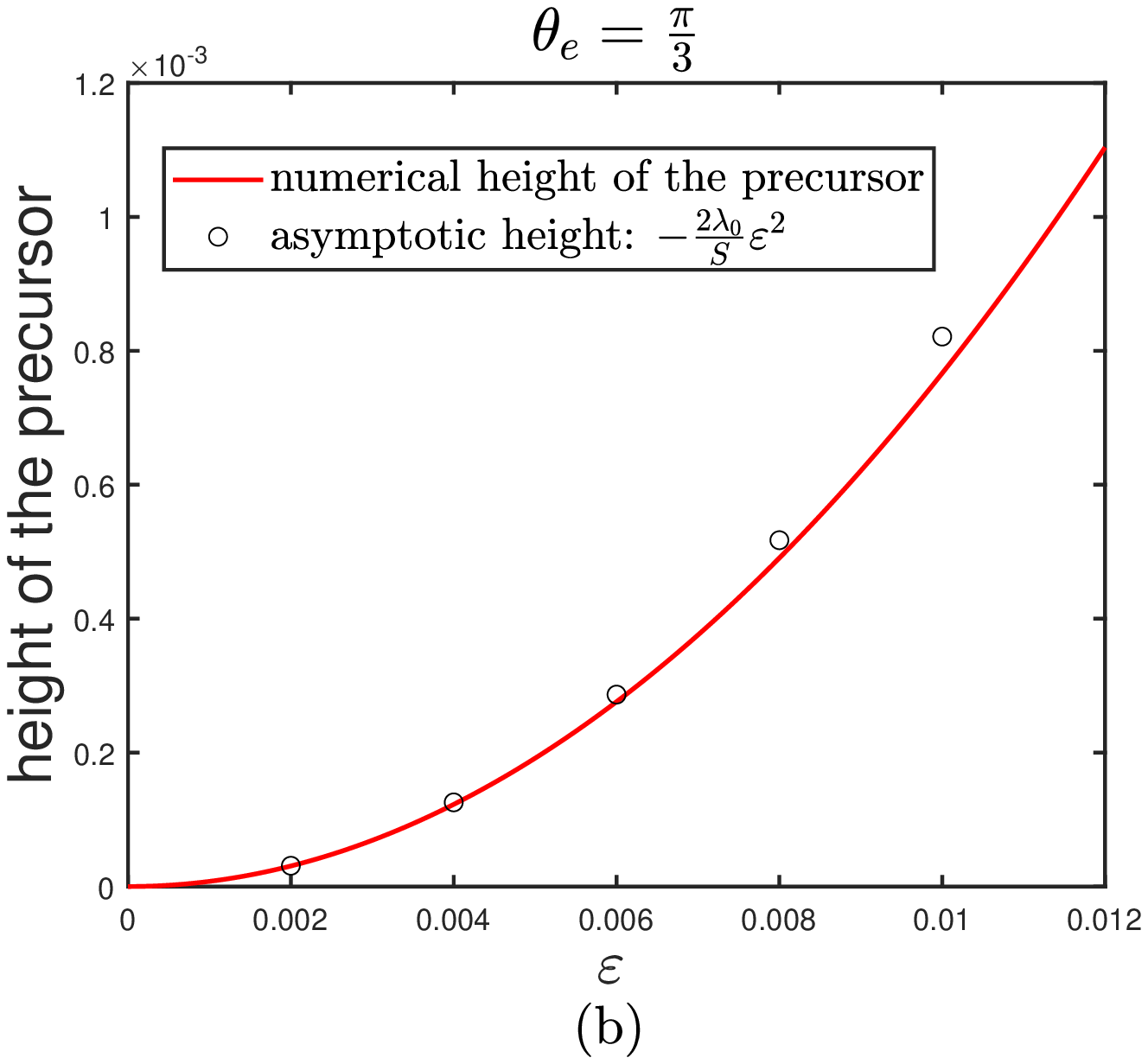}}
    \caption{\label{heightlog}~The height of the precursor of the regularized model as a function of the parameter $\varepsilon$ when Young's angle $\theta_e={\pi}/{3}$, where (a) the log-log plot indicates that it goes to zero at the second-order, i.e., $\mathcal{O}(\varepsilon^2)$; and (b) shows that it agrees perfectly with the formula \eqref{eq:height} given by the asymptotic analysis.}
\end{figure}

Fig. \ref{numerical_inner} plots the convergence results from the profile of the numerical solutions near the left contact point to that of the inner solution given by
Eq.~\eqref{eq:inner} (shown in red line) from our asymptotic analysis under Young's angle $\theta_e={\pi}/{3}$, when we decrease $\varepsilon$.
As shown by the figure, the numerical solutions again agree well with our asymptotic analysis. Finally, Tables \ref{tab:contact_angle}-\ref{tab:contact_point} respectively show the convergence results of the apparent contact angles and the contact point locations when we decrease $\varepsilon$, where the apparent contact angle and the position
of the apparent contact point are determined by a curve fitting. From these tables, we clearly observe the monotone convergence results.
\begin{figure}[htb]
	\centering
	\includegraphics[width=10.0cm]{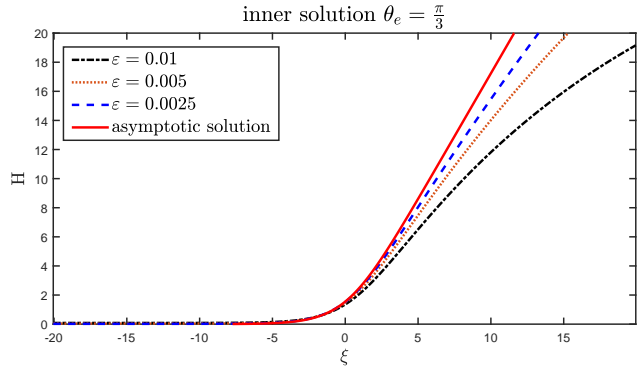}
	\caption{Comparisons between the inner solution of the profile given by Eq.~\eqref{eq:inner} (in red line) from our asymptotic analysis and the numerical solutions for different $\varepsilon=0.01,0.005,0.0025$ under Young's angle $\theta_e={\pi}/{3}$. $\xi$ and $ H $ are rescaled inner variables.}
	\label{numerical_inner}
\end{figure}

\begin{table}[htb]
	\centering
	\fontsize{7}{10}\selectfont
	\caption{Convergence results from the apparent contact angles produced by the regularized model to two different Young's contact angles $\theta_e={\pi}/{3}$ (left) and $\theta_e={\pi}/{6}$ (right), respectively, when we gradually reduce $\varepsilon$.}
	\label{tab:contact_angle}
	\begin{tabular}{|c|c|c|c|c|}
		\hline
		\multirow{2}{*}{$\varepsilon$}&
		\multicolumn{2}{c|}{$\theta_e={\pi}/{3}\approx1.0471$}&\multicolumn{2}{c|}{ $\theta_e={\pi}/{6}\approx0.5235$}\cr\cline{2-5}
		&Apparent contact angle&error&Apparent contact angle&error\cr
		\hline
		\hline
		0.01&1.0349&1.2231e-2&0.5105&1.3016e-2\cr\hline
		0.008&1.0397&7.4144e-3&0.5162&7.3774e-3\cr\hline
		0.006&1.0432&3.9821e-3&0.5198&3.7844e-3\cr\hline
		0.004&1.0455&1.6858e-3&0.5220&1.5679e-3\cr\hline
		0.002&1.0469&2.9315e-4&0.5232&3.7800e-4\cr
		\hline
	\end{tabular}
\end{table}
\begin{table}[htb]
	\centering
	\fontsize{7}{10}\selectfont
	\caption{Convergence results from the locations of apparent contact points by the regularized model to the exact location of contact point by the original model
     when we gradually reduce $\varepsilon$ under two different Young's angles $\theta_e={\pi}/{3}$ (left) and $\theta_e={\pi}/{6}$ (right), respectively.}
	\label{tab:contact_point}
	\begin{tabular}{|c|c|c|c|c|}
		\hline
		\multirow{2}{*}{$\varepsilon$}&
		\multicolumn{2}{c|}{$\theta_e={\pi}/{3}$ (Exact location$\approx 0.45210$) }&\multicolumn{2}{c|}{ $\theta_e={\pi}/{6}$ (Exact location$\approx 0.67966$)}\cr\cline{2-5}
		&Apparent contact point&error&Apparent contact point&error\cr
		\hline
		\hline
		0.01&0.45568&3.5848e-3&0.68891& 9.2526e-3\cr\hline
		0.008&0.45427&2.1688e-3&0.68486&5.2044e-3\cr\hline
		0.006&0.45326&1.1643e-3&0.68232&2.6568e-3\cr\hline
		0.004&0.45259&4.9416e-4&0.68076&1.0975e-3\cr\hline
		0.002&0.45219&8.8512e-5&0.67993&2.6869e-4\cr
		\hline
	\end{tabular}
\end{table}

\section{Conclusion and discussion}

In this paper, we proposed a new regularized variational model for simulating wetting/dewetting problems by introducing a height-dependent surface energy density $\gamma^{\varepsilon}(h)$ defined in~\eqref{eqn:reg}. By our asymptotic analysis and numerical simulations, we showed that the regularized term $\gamma^{\varepsilon}(h)$
will lead to that a bare substrate, resulted from the original model~\eqref{eqn:func}, is always covered by a precursor layer whose thickness depends on the small parameter $\varepsilon$. The new regularized model transforms the original free boundary problem~\eqref{eqn:func} into a fixed domain problem, which brings many advantages in both analysis and simulations.
As a first step, we focus on the equilibrium state of wetting/dewetting problems in this paper by using the new
regularized model, and study the equilibrium problem from the aspects of positivity-preserving property of its solutions, asymptotic limits, and $\Gamma$-convergence to the classical sharp-interface model when $\varepsilon$ goes to zero.

Another important contribution of this paper is devoted to investigating the connections between the new regularized model and the classical sharp-interface model in the limit $\varepsilon\to 0$. Under proper conditions for the interpolation function $g(z)$, we show by matched asymptotic analysis that the regularized model is a singular perturbation problem, and its asymptotic limit yields a spherical cap shape which intersects with the substrate at an apparent contact angle $\theta_e$. This recovers with the well-known
Young's equation. Moreover, the height of the precursor asymptotically approaches zero at the second-order rate, i.e., $\mathcal{O}(\varepsilon^2)$. These theoretical analysis results are validated by numerically solving the regularized model through the modified Newton's method. Last but not least, we rigorously prove that the regularized functional $\Gamma$-converges to the classical sharp-interface functional in $L^1$ space, under some mild and physically meaningful conditions.

Compared with the existing models for simulating wetting/dewetting problems, we note that the new regularized model has the following several advantages that we will further demonstrate in the forthcoming papers: (1) it only needs to be solved in a fixed domain, no matter what kind of problems (i.e., equilibrium or kinetic problems) we focus on; (2) it does not need to explicitly handle the contact line motion, so it can automatically capture topological change events; (3) it is one dimension lower than the phase-field model, so its computational cost is relatively smaller. In addition, we can make use of these advantages and extend our regularization approach to the study of other thin film models, e.g., lubrication model for liquid-state wetting problems.

In the future, we will use the new regularized model to simulate dynamic problems of wetting/dewetting, including $L^2$-gradient flow with the volume constraint
and $H^{-1}$ gradient flow for surface diffusion.\cite{Jiang22Huang} In particular, the regularity, positivity-preserving property and asymptotic limits of its kinetic solutions, the contact line dynamics, and the efficient numerical algorithms will be our future concerns.
It is straightforward and very interesting to include anisotropic surface energy effects into the regularized model for solid-state wetting/dewetting,\cite{Jiang22Huang}
and include elastic effects for simulating epitaxial growth of thin solid films,\cite{Fonseca2007equ,Davoli2019,Boccardo2022} and etc.
In addition, we only study the partial wetting/dewetting regime in this paper, i.e., Young's angle $\theta_e\in (0,{\pi}/{2})$.
If $\theta_e$ is no less than $\pi/2$, we can no longer use the height function $h(x)$ to represent the film/vapor interface curve,
but similar discussions can also be performed by using other parameterizations.

\appendix

\setcounter{definition}{0}
\renewcommand{\thedefinition}{A.\arabic{definition}}
\setcounter{lemma}{0}
\renewcommand{\thelemma}{A.\arabic{lemma}}
\section{Bounded variation function}\label{bv_preliminary}
In this appendix, we briefly introduce some definitions and properties of the bounded variation function. For more details, we refer the reader to Refs. \refcite{ambrosio2000} and \refcite{evans1991}. In this appendix, we always assume $\Omega \subset \mathbb{R}^{n}$.

\begin{definition}{(Bounded variation function)}
	Let $u \in L^{1}(\Omega)$. $u$ is a bounded variation function in $\Omega$ if the distributional derivative of $u$ is represented by a finite Radon measure, i.e.
	$$
	\int_{\Omega} u \frac{\partial \psi}{\partial x_{i}} \mathrm{d}x=-\int_{\Omega} \psi \mathrm{d}D_{i} u \quad \forall \psi \in C_{0}^{\infty}(\Omega), \quad i=1, \ldots, n
	$$
	for some $\mathbb{R}^{n}$-valued Radon measure $D u=\left(D_{1} u, \ldots, D_{n} u\right)$ in $\Omega$. The vector space of all bounded variation function is denoted by $BV(\Omega)$.
\end{definition}
We can always write $D u=\sigma|D u|$, where $|D u|$ is a positive Radon measure and $\sigma=\left(\sigma_{1}, \ldots, \sigma_{n}\right)$ with $|\sigma(x)|=1$ for $|D u|$-a.e. $x \in \Omega$.

Moreover, we say that $ u\in BV_{\mathrm{loc}}(\Omega) $  if $  u \in BV(\Omega_0) $ for every $ \Omega_0\subset\subset\Omega $, i.e. every open $\Omega_0$ with $ \overline{\Omega_0} $ compact and contained in $\Omega$.


\begin{definition}{(Perimeter)}
	Let $ U $ be a Borel set and $\Omega$ be an open set in $ \mathbb{R}^n $. Define the perimeter of $ U $ in $\Omega$ as
	\[
		P(U,\Omega):=\int_\Omega|D\chi_{_{U}}|,
	\]
	where $\chi_{_{U}}$ is the characteristic function of $ U $.
\end{definition}

\begin{definition}{(Approximate jump points)}\label{def_jump}
	Let $u \in L_{\mathrm{loc}}^{1}(\Omega)$ and $x \in \Omega$. We say that $x$ is an approximate jump point of $u$ if there exist $a, b \in \mathbb{R}$ and an $ (n-1) $-dimensional unit vector $\boldsymbol{\nu}$ such that $a \neq b$ and
	$$
	\lim _{R \downarrow 0} \int_{B_{R}^{+}(x,\boldsymbol{\nu} )}R^{-n}|u(y)-a| \mathrm{d}y=0, \quad \lim _{R \downarrow 0} \int_{B_{R}^{-}(x, \boldsymbol{\nu})}R^{-n}|u(y)-b| \mathrm{d}y=0 ,
	$$
	where $ u^+:=a $ and $ u^-:=b $ are called one-side approximate limits and
	\[
		\begin{array}{l}
			B_{R}^{+}(x, \boldsymbol{\nu}):=\left\{y \in B_{R}(x):\langle y-x, \boldsymbol{\nu}\rangle>0\right\}, \\
			B_{R}^{-}(x, \boldsymbol{\nu}):=\left\{y \in B_{R}(x):\{y-x, \boldsymbol{\nu}\rangle<0\right\},
		\end{array}
	\]
	which means two half balls contained in $ B_{R}(x) $ determined by $ \boldsymbol{\nu} $. The set of approximate jump points is denoted by $J_u$.
\end{definition}

We recall the usual decomposition
$$
D u=\nabla u \mathcal{L}^{n}+\left(u^{+}-u^{-}\right) \otimes \boldsymbol{\nu}_{u} \mathcal{H}^{n-1}\mathrm{~\llcorner} J_u+D^{c} u,
$$
where $ Du $ is the distribution derivative of $u$, $\nabla u$ is the Radon-Nikodym derivative of $D u$ with respect to the Lebesgue measure $\mathcal{L}^n$, $ \boldsymbol{\nu}_{u} $ is unit normal vector, $ \mathcal{H}^{n-1}\mathrm{~\llcorner} J_u $ means Hausdorff measure restricted to the set $ J_u $, and $D^{c} u$ is the Cantor part of $D u$. For the sake of simplicity, we denote $ D^su:= \left(u^{+}-u^{-}\right) \otimes \boldsymbol{\nu}_{u} \mathcal{H}^{n-1}\mathrm{~\llcorner} J_u+D^{c} u$, which called the singular part.

%




We introduce two important properties by two lemmas whose proof can be seen in Ref. \refcite{ambrosio2000}.
\begin{lemma}{(Property of $D^{a} u, D^{s} u$)}\label{prop_singular}
	Let $u \in B V(\Omega)$, then $D^{a} u=D u\llcorner(\Omega \backslash S)$ and $D^{s} u=D u\llcorner S$, where
	$$
	S:=\left\{x \in \Omega: \lim _{R \downarrow 0} R^{-n}|D u|\left(B_{R}(x)\right)=\infty\right\}.
	$$
\end{lemma}

\begin{lemma}{(Property of $D^{c} u$)}\label{cantor_lemma}
	Let $u \in B V(\Omega)$, and let $B$ be a Borel set with its $ (n-1) $-dimensional Hausdorff measure $\mathcal{H}^{n-1}(B)<+\infty$. Then $\left|D^{c} u\right|(B)=0$.
\end{lemma}

From Ref. \refcite{demengel1984}, we know a special decomposition:
\begin{equation}\label{area_formula_decompositon}
	\int_{\Omega}\sqrt{1+|Du|^2}:=\int_{\Omega}\sqrt{1+|\nabla u|^2}\mathrm{d}x+|D^{c} u|(\Omega)+\int_{J_u\cap\Omega}|u^+-u^-|  \mathrm{~d} \mathcal{H}^{n-1},
\end{equation}
whose geometric meaning is the perimeter of the subgraph $ U:=\{(x,t)\in\Omega\times\mathbb{R}:t<h(x)\} $ of $ h $ in $\Omega$.\cite{giusti1984} Especially, if $ h $ is Lipschitz continuous, it represents the area of the surface $ \{(x,h(x)):x\in\Omega\} $.


At last, we introduce the compactness of the bounded variation functions.\cite{evans1991}
\begin{lemma}{(Compactness of BV function)}\label{compact_lemma}
	Let $\Omega \subset \mathbb{R}^{n}$ be open and bounded, with $\partial \Omega$ Lipschitz. Assume $\left\{f_{j}\right\}_{j=1}^{\infty}$ is a sequence in $B V(\Omega)$ satisfying
	$$
	\sup _{j}\left\|f_{j}\right\|_{B V(\Omega)}<\infty .
	$$
	Then there exists a subsequence $\left\{f_{j_{k}}\right\}_{k=1}^{\infty}$ and a function $f \in B V(\Omega)$ such that
	$$
	f_{j_{k}} \rightarrow f \text { in } L^{1}(\Omega),
	$$
	as $k \rightarrow \infty$.
\end{lemma}

\setcounter{lemma}{0}
\renewcommand{\thelemma}{B.\arabic{lemma}}
\section{Proof of Lemma \ref{lemma_upper}}\label{proof4.2}

Firstly, we recall the definition and properties of molliﬁers.\cite{giusti1984}\\
A function $\eta(x)$ is called a mollifier if
\begin{itemize}
\item[(i)] $\eta(x) \in C_{0}^{\infty}\left(\mathbb{R}^{n}\right)$,
\item[(ii)] $\eta$ is zero outside a compact subset of $B_{1}=\left\{x \in \mathbb{R}^{n}:|x|<1\right\}$,
\item[(iii)] $\int \eta(x) d x=1$.
\end{itemize}
If in addition we have
\begin{itemize}
\item[(iv)] $\eta(x) \geqslant 0$,
\item[(v)] $\eta(x)=\mu(|x|)$ for some function $\mu: \mathbb{R}^{+} \rightarrow \mathbb{R}$,
\end{itemize}
then $\eta$ is a positive symmetric mollifier. An example of positive symmetric mollifiers is the function
$$
\eta(x)= \begin{cases}0 & |x| \geqslant 1, \\ C \exp \left(\frac{1}{|x|^{2}-1}\right) & |x|<1,\end{cases}
$$
where $C$ is a normalizing constant such that $\int \eta(x) d x=1$.

Given a positive symmetric mollifier $\eta$ and a function $f \in L_{\text {loc }}^{1}\left(\mathbb{R}^{n}\right)$, define for each $\varepsilon>0$
$$
\eta_{\varepsilon}(x):=\varepsilon^{-n} \eta\left(\frac{x}{\varepsilon}\right),
$$
$$
f_{\varepsilon}(x):=(\eta_{\varepsilon} * f)(x)=\varepsilon^{-n} \int_{\mathbb{R}^{n}} \eta\left(\frac{x-z}{\varepsilon}\right) f(z) \mathrm{d}z=\int_{\mathbb{R}^{n}} \eta(w) f(x+\varepsilon w) \mathrm{d}w.
$$
It is straightforward to have the following two properties of mollifiers:

(a) $ f(x) \geqslant 0\ a.e. $ in $\Omega$ $\  \Rightarrow f_{\varepsilon}(x) \geqslant 0 $ if $\varepsilon$ is small enough,

(b) $\operatorname{Supp} f \subseteq A \Rightarrow \operatorname{Supp} f_{\varepsilon} \subseteq A_{\varepsilon}=\{x: \operatorname{dist}(x, A) \leqslant \varepsilon\}$.

Using the technique of mollifiers, Lemma B.1 of Ref. \refcite{bildhauer03} can be established.
\begin{lemma}
	Let $h \in B V\left(\Omega\right)$. Then there is a sequence $\left\{h_{j}\right\}$ in $C^{\infty}\left(\Omega\right)$ satisfying
	\begin{equation*}
		\begin{aligned}
			&\lim _{j \rightarrow \infty} \int_{\Omega}\left|h_{j}-h\right| \mathrm{d} x=0 \\
			&\lim _{j \rightarrow \infty} \int_{\Omega} \sqrt{1+\left|\nabla h_{j}\right|^{2}} \mathrm{~d} x=\int_{\Omega} \sqrt{1+|D h|^{2}}
		\end{aligned}
	\end{equation*}
	Moreover, the trace of each $h_{j}$ on $\partial \Omega$ coincides with the trace of $h$.
\end{lemma}

Based on this lemma, the remaining problem is to show $ h_j\geqslant0 $ and $ \lim_{m\to\infty}\mathcal{L}^n(\operatorname{Supp}(h_j))\leqslant\mathcal{L}^n(\operatorname{Supp}(h)). $

From the proof of Lemma B.1 in Ref. \refcite{bildhauer03}, we know the recovery sequence is
\[
h_{j}=\sum_{i=1}^{\infty} \eta_{\varepsilon_{j,i}} \ast\left(\varphi_{i} h\right) ,
\]
where $\{\varepsilon_{j,i}\}$ are sufficiently small. $ \{\varphi_{i}\} $ is a partition of unity, i.e. for a covering $ \{U_i\} $ of $ \Omega $,
\[
	\varphi_{i} \in C_{0}^{\infty}\left(U_{i}\right), \quad 0 \leqslant \varphi_{i} \leqslant 1, \quad \sum_{i=1}^{\infty} \varphi_{i}=1.
\]

Since $\varphi_i\geqslant0$ and $ h\geqslant0 \ a.e.$ in $\Omega$, from property (a) of the mollifiers, we know $ h_j\geqslant0 $.

For any $\delta>0$, we can select a sequence $\{\varepsilon_{j,i}\}$ such that $ \varepsilon_{j,i}<\delta $. From property (b) of the mollifiers and $ h\geqslant0 \ a.e.$ in $\Omega$,
\begin{align*}
	 \lim\limits_{j\to\infty}\mathcal{L}^n(\operatorname{Supp}(h_j))&= \lim\limits_{j\to\infty}\mathcal{L}^n\bigg(\operatorname{Supp}(\sum_{i=1}^{\infty} \eta_{\varepsilon_{j,i}} \ast\left(\varphi_{i} h\right))\bigg)\\
	&= \lim\limits_{j\to\infty}\mathcal{L}^n\bigg(\cup_i\operatorname{Supp}( \eta_{\varepsilon_{j,i}} \ast\left(\varphi_{i} h\right))\bigg)\\
	&\leqslant \lim\limits_{j\to\infty}\mathcal{L}^n\big(\cup_i\{x:\operatorname{dist}(x,\operatorname{Supp}(\varphi_{i}h))<\delta\} \big)\\
	&\leqslant \lim\limits_{j\to\infty}\mathcal{L}^n\big(\{x:\operatorname{dist}(x,\operatorname{Supp}(h))<\delta\}\big)\\
	&=\mathcal{L}^n\big(\{x:\operatorname{dist}(x,\operatorname{Supp}(h))<\delta\}\big),
\end{align*}
Let $\delta\to0$,
\begin{equation*}
	\lim_{j\to\infty}\mathcal{L}^n(\operatorname{Supp}(h_j))\leqslant\mathcal{L}^n(\operatorname{Supp}(h)).
\end{equation*}

\section*{Acknowledgments}
The authors gratefully acknowledge many helpful discussions with Linlin Su (Southern University of Science and Technology) during the preparation of the paper.
The work of Zhen Zhang was partially supported by the NSFC grant (No. 11731006) and (No. 12071207), NSFC Tianyuan-Pazhou grant (No. 12126602), the Guangdong Basic and Applied Basic Research Foundation (2021A1515010359), and the Guangdong Provincial Key Laboratory of Computational Science and Material Design (No. 2019B030301001). The work of Wei Jiang was supported by the National Natural Science Foundation of China Grant (No.~11871384).

\section*{Authors' Contributions}
Z. Zhou and Z. Zhang did the mathematical analysis. W. Jiang and Z. Zhang designed and coordinated the project. All participated on the preparation of the manuscript. All authors gave final approval for publication.

\bibliographystyle{ws-m3as}
\bibliography{Reference}
\end{document}